\documentclass[11pt]{amsart}
\usepackage{amsmath,amscd,amssymb,amsfonts,amsthm}
\usepackage[cmtip,matrix,arrow]{xy}

\newtheorem{theorem}{Theorem}[section]
\newtheorem{corollary}[theorem]{Corollary}
\newtheorem{proposition}[theorem]{Proposition}

\newtheorem{lemma}[theorem]{Lemma}
\theoremstyle{definition}    
\newtheorem{definition}[theorem]{Definition}
\theoremstyle{remark}

\newtheorem{remark}[theorem]{Remark}
\newtheorem{remarks}[theorem]{Remarks}
\newtheorem{example}[theorem]{Example}
\newtheorem{examples}[theorem]{Examples}

\newcommand\A{\mathcal{A}}

\newcommand{\Cour}[1]      {[\![#1]\!]}

\newcommand\G{\mathcal{G}}

\renewcommand{\L}{\mathcal{L}}

\newcommand{\T}{\mathbb{T}}

\newcommand{\ca}{\mathcal}

\newcommand{\U}{\on{U}}
\newcommand{\E}{\ca{E}}

\newcommand{\R}{\mathbb{R}}
\newcommand{\C}{\mathbb{C}}

\newcommand{\Z}{\mathbb{Z}}

\newcommand\pt{\on{pt}}

\newcommand{\dd}{\mathfrak{d}}
\newcommand{\n}{\mathfrak{n}}

\newcommand\lie[1]{\mathfrak{#1}}
\renewcommand{\k}{\lie{k}}
\newcommand{\h}{\lie{h}}
\newcommand{\g}{\lie{g}}

\renewcommand{\a}{\mathsf{a}}

\newcommand{\on}{\operatorname}

\newcommand{\Aut}{ \on{Aut} } 
\newcommand{\Ad}{ \on{Ad} }
\newcommand{\ad}{\on{ad}}

\renewcommand{\ker}{ \on{ker}}

\newcommand{\Mult}{  \on{Mult}}

\newcommand{\da}{\dasharrow}
\newcommand{\sz}{\mathsf{s}}
\newcommand{\tz}{\mathsf{t}}

\newcommand\qu{/\kern-.7ex/} 

\newcommand{\hra}{\hookrightarrow}
\newcommand{\ra}{\rightarrow}

\renewcommand{\d}{{\mbox{d}}}
\newcommand{\ol}{\overline}

\newcommand\sig{\sigma}
\newcommand\eps{\epsilon}
\newcommand\Om{\Omega}
\newcommand\om{\omega}

\newcommand{\f}{\frac}

\newcommand{\p}{\partial}
\renewcommand{\l}{\langle}
\renewcommand{\r}{\rangle}
\newcommand\hh{{\f{1}{2}}}

\newcommand{\eeq}{\end{eqnarray*}}
\newcommand{\beq}{\begin{eqnarray*}}

\newcommand{\pr}{\on{pr}}

\newcommand{\wt}{\widetilde}
\newcommand{\mf}{\mathfrak}
\newcommand{\rra}{\rightrightarrows}
\newcommand{\GL}{\on{GL}}

\renewcommand{\subset}{\subseteq}
\begin{document}
\sloppy
\title{Poisson Geometry from a Dirac perspective}
\author{Eckhard Meinrenken}
\maketitle
\begin{abstract}
We present proofs of classical results in Poisson geometry using techniques from Dirac geometry. This article is based on mini-courses at the Poisson summer school in Geneva, June 2016, and at the workshop \emph{Quantum Groups and Gravity} at the University of Waterloo, April 2016. 
\end{abstract}
\medskip\medskip

 \bigskip


\tableofcontents
These notes are based on a 4-hour mini-course for the summer school \emph{Poisson 2016}  in Geneva, June 27-July 2, 2016, as well as a 2-hour mini-course for the workshop \emph{Quantum Groups and Gravity} at the University of Waterloo, April 25--29, 2016.   Accordingly, the first four sections will follow the lectures in Geneva, while the last two sections are loosely based on the lectures in Waterloo. 

A general theme of these lectures was to present some of the classical results in Poisson geometry within the larger framework of \emph{Dirac geometry}. That is, besides the familiar viewpoints of Poisson structures as brackets $\{\cdot,\cdot\}\colon C^\infty(M)^2\to C^\infty(M)$ on the algebra of functions,  or as bivector fields $\pi\in\mf{X}^2(M)$ on the manifold, we will also consider  
Poisson structures as Dirac structures, that is, maximal isotropic subbundles $E\subset TM\oplus T^*M$ satisfying a certain integrability condition.  One advantage gained by the Dirac-geometric approach is more flexibility -- for example, Dirac structures can be pulled back under smooth maps satisfying a transversality condition, whereas Poisson structures can be pulled back only in very special cases.  Putting this flexibility to use, we will present the recent Dirac-geometric proofs of the  \emph{Weinstein splitting theorem} (after Bursztyn-Lima-M), the \emph{Karasev-Weinstein symplectic realization theorem} (after Crainic-M\u{a}rcu\c{t} and 
Frejlich-M\u{a}rcu\c{t}), and the \emph{Drinfeld theorems} for the classification of Poisson Lie groups and Poisson homogeneous spaces (after Liu-Weinstein-Xu, with some additional ideas from Patrick Robinson and myself). 

Following the nature of the mini-courses, this article is an essentially self-contained introduction to the foundations of Poisson geometry, but at a brisk pace and leaving out a number of important topics. Among other things, we will not discuss Schouten brackets, Poisson homology and cohomology, Poisson submanifolds and coisotropic submanifolds,  and stability theorems. We refer to 
the textbooks \cite{ca:ge}, \cite{duf:po}, \cite{lau:po},  \cite{mac:gen}, \cite{va:po} for details on these topics. 
\vskip.2in 

\noindent{\it Notation and conventions.} We will be working in the $C^\infty$ category of smooth manifolds, although most of our constructions work equally well in complex holomorphic category. For a manifold $M$, we denote by $\mf{X}(M)$ the Lie algebra of smooth vector fields on $M$, and by $\Omega(M)$ the algebra of differential forms. For a smooth vector bundle $E\to M$, we denote by $\Gamma(E)$ the $C^\infty(M)$-module of smooth sections. The Lie algebra of a Lie group is defined using \emph{left-invariant} vector fields; the same convention is used for Lie groupoids. See Appendix A for a justification, and further sign conventions. 
\vskip.4in

\noindent{\bf Acknowledgements}. I would like to thank the organizers of the Poisson 2016 summer school, Anton Alekseev and Maria Podkopaeva (Geneva), for providing an excellent environment, and for encouragement to write this notes. I also thank Pedro Frejlich, Ioan M\u{a}rcu\c{t},  Mykola Matviichuk, Alan Weinstein, Ping Xu, as well as the referees for helpful comments. 
\vskip.4in

\section{Poisson manifolds}\label{sec:poissonmanifolds}
\subsection{Basic definitions}\label{subsec:poisson}
Poisson structures on manifolds can be described in several equivalent ways. The quickest definition is in terms of a bracket operation on smooth functions. 

\begin{definition}\cite{lic:var}
A \emph{Poisson structure} on a manifold $M$ is a bilinear map 
\[ \{\cdot,\cdot\}\colon C^\infty(M)\times C^\infty(M)\to C^\infty(M)\]
with the properties 
\begin{eqnarray}
\label{eq:1}
\{f,g\}&=&-\{g,f\}\\
\label{eq:3}\{f,gh\}&=&\{f,g\}h+g\{f,h\}\\
\label{eq:2}\{f,\{g,h\}\}&=&\{\{f,g\},h\}+\{g,\{f,h\}\}
\end{eqnarray}
for all $f,g,h\in C^\infty(M)$. A map $\Phi\colon N\to M$ between Poisson manifolds is a \emph{Poisson map} if the pull-back map $\Phi^*\colon C^\infty(M)\to C^\infty(N)$ 
preserves brackets. 
\end{definition}

Condition \eqref{eq:3} states that for all $f\in C^\infty(M)$, the operator $\{f,\cdot\}$ is a derivation of the algebra of smooth functions $C^\infty(M)$, that is, a vector field. One calls 
\[ X_f=\{f,\cdot\}\]
the \emph{Hamiltonian vector field} determined by the \emph{Hamiltonian} $f$. In various physics interpretations, the flow of $X_f$ describes the dynamics of a system with Hamiltonian $f$.   

Suppose $\{\cdot,\cdot\}$ is a bilinear form on $C^\infty(M)$ satisfying \eqref{eq:1} and \eqref{eq:3}. Then $\{\cdot,\cdot\}$ is a derivation in both arguments.  A little less obviously, the Jacobiator 
$\on{Jac}(\cdot,\cdot,\cdot)$ for the bracket, given by 
\begin{equation}\label{eq:jac}
\on{Jac}(f,g,h)=\{f,\{g,h\}\}+ \{g,\{h,f\}\}+ \{h,\{f,g\}\}
\end{equation} 
is also a derivation in each argument (exercise). It follows the values of 
$\{f,g\}$ and of $\on{Jac}(f,g,h)$ at any given point $m\in M$ depend only on the exterior differentials of $f,g,h$ at $m$.  Letting 
\[ \mf{X}^\bullet(M)=\Gamma(\wedge^\bullet TM)\] 
be the $\Z$-graded algebra (under pointwise wedge product) of multi-vector fields, we 
obtain a 2-vector field and a 3-vector field
\begin{eqnarray*}
\pi\in \mf{X}^2(M),&& \pi(\d f,\d g)=\{f,g\}\\
\Upsilon_\pi\in \mf{X}^3(M),&& \Upsilon_\pi (\d f,\d g,\d h)=
\on{Jac}(f,g,h). \end{eqnarray*}
The Jacobi identity \eqref{eq:2} is thus equivalent to $\Upsilon_\pi=0$. In particular, to check if $\{\cdot,\cdot\}$ satisfies the Jacobi identity, it is enough to check on functions whose differentials span $T^*M$ everywhere.
\begin{remark}
In terms of the \emph{Schouten bracket} of multi-vector fields, the 3-vector field 
$\Upsilon_\pi$ associated to a bivector field $\pi$ is given by $\Upsilon_\pi=-\hh [\pi,\pi]$. Thus, $\pi$ defines a Poisson structure if and only if 
$[\pi,\pi]=0$. In these notes, we we will not use the Schouten bracket formalism.
\end{remark}

\subsection{Examples}

\begin{example}\label{ex:1}
The standard Poisson bracket on `phase space' $\R^{2n}$, with coordinates $q^1,\ldots,q^n$ and $p_1,\ldots,p_n$, is given by 
\[ \{f,g\}=\sum_{i=1}^n \f{\p f}{\p q^i}\f{\p g}{\p p_i}-\f{\p f}{\p p_i}\f{\p g}{\p q^i}.\]
To verify \eqref{eq:2}, it suffices to check on \emph{linear} functions; but 
in that case  all three terms in the definition of $\on{Jac}$ are zero. (E.g., 
$\{f,\{g,h\}\}=0$ since $\{g,h\}$ is constant.) The differential equations defined by the Hamiltonian vector field $X_f=\{f,\cdot\}$ are 
%
%
\emph{Hamilton's equations}

\[ \dot{q}^i=\f{\p f}{\p p_i},\ \ \dot{p}_i=-\f{\p f}{\p q^i}
\]
from classical mechanics. (Our sign conventions (cf.~ Appendix \ref{app:signs}) are such that a vector field $X=\sum_j a^j(x) \f{\p}{\p x^i}$ corresponds to the ODE $\f{d x^j}{d t}=-a^j\big(x(t)\big)$.) The Poisson bivector field is
\begin{equation}\label{eq:constant}
 \pi=\sum_{i=1}^n  \f{\p}{\p q^i}\wedge \f{\p}{\p p_i}.
\end{equation}
More generally, any symplectic manifold $(M,\omega)$ becomes a Poisson manifold, 
in such a way that the Hamiltonian vector fields $X_f=\{f,\cdot\}$ satisfy
$ \om(X_f,\cdot)=-\d f$. In local symplectic coordinates $q^1,\ldots,q^n,p_1,\ldots,p_n$, with $\omega=\sum_i \d q^i\wedge \d p_i$, the Poisson structure is given by the formula \eqref{eq:constant} above. Note that with our sign conventions, the maps $\pi^\sharp\colon T^*M\to TM,\ \mu\mapsto \pi(\mu,\cdot)$ and $\omega^\flat\colon TM\to T^*M,\ v\mapsto \omega(v,\cdot)$ are related by 
\[ \pi^\sharp=-(\omega^\flat)^{-1}.\]
\end{example}

\begin{example}
If $\dim M=2$, then \emph{any} bivector field $\pi\in\mf{X}^2(M)$ is  Poisson: The vanishing of 
$\Upsilon_\pi$ follows because on a 2-dimensional manifold, every 3-vector field is zero. 
\end{example}

\begin{example}\label{ex:liepoisson}
For any Lie algebra $\g$, the dual space $\g^*$ has a Poisson structure known as the \emph{Lie-Poisson} structure or \emph{Kirillov-Poisson} structure. For all $\mu\in \g^*$, identify $T_\mu^*\g^*=\g$. Then 
\[ \{f,g\}(\mu)=\l \mu,\ [\d_\mu f,\ \d_\mu g]\r.\]
The corresponding bivector field is 
\[ \pi=\hh \sum_{ijk} c_{ij}^k\,\, \mu_k \, \f{\p}{\p \mu_i}\wedge \f{\p}{\p \mu_j},\]
where $\mu_i$ are the coordinates on $\g^*$, relative to a basis $\eps_1,\eps_2,\ldots,\epsilon_n$ of $\g$, and $c_{ij}^k$ the corresponding structure constants, defined by $[\eps_i,\eps_j]=\sum_k c_{ij}^k \eps_k$. 
To verify the Jacobi identity for $\{\cdot,\cdot\}$, it suffices to check on \emph{linear} functions, since the differentials of linear functions span $T^*\g^*$ everywhere. 
For every $\xi\in\g$, let $\phi_\xi(\mu)= \l\mu,\xi\r$.
Then $\{\phi_\xi,\phi_\zeta\}=\phi_{[\xi,\zeta]}$ for all $\xi,\zeta\in\g$. Hence, the Jacobi identity for $\{\cdot,\cdot\}$ on linear functions 
is just the Jacobi identity for the Lie bracket $[\cdot,\cdot]$ of $\g$. 

Conversely, given a Poisson structure on a finite-dimensional vector space $V$, where $\{\cdot,\cdot\}$ is \emph{linear} in the sense that the coefficients of the Poisson tensor are linear, one obtains a Lie algebra structure on $\g:=V^*$, with $\{\cdot,\cdot\}$ as the corresponding Lie-Poisson structure. This gives a 1-1 correspondence 
\begin{equation}\label{eq:correspondence1}
\Big\{\begin{tabular}{c} Vector spaces with \\
linear Poisson structures\end{tabular}\Big\}
\stackrel{1-1}{\longleftrightarrow} 
\Big\{\mbox{ Lie algebras}\Big\}.
\end{equation}
\end{example}
\vskip.2in 

\subsection{Lie algebroids as Poisson manifolds}
The correspondence \eqref{eq:correspondence1} extends to vector bundles, with Lie algebras replaced by Lie \emph{algebroids}.  
\begin{definition}
A \emph{Lie algebroid} $(E,\a,[\cdot,\cdot])$ over $M$ is a vector bundle $E\to M$, together with a bundle map $\a\colon E\to TM$ called the \emph{anchor} and a 
Lie bracket $[\cdot,\cdot]$ on its space $\Gamma(E)$ of sections, such that for all $\sigma,\tau\in \Gamma(E)$ and $f\in C^\infty(M)$, 
\begin{equation}\label{eq:leibnitz}
 [\sigma,f\tau]=f[\sigma,\tau]+\big(\a(\sigma)(f)\big)\ \tau.
 \end{equation}
\end{definition}
Some examples: 
\begin{itemize}
\item $E=TM$ is a Lie algebroid, with anchor the identity map. More generally, the tangent bundle to a regular foliation is a Lie algebroid, with anchor the inclusion. 
\item A Lie algebroid over $M=\pt$ is the same as a finite-dimensional Lie algebra $\g$. 
\item Given a $\g$-action on $M$, the trivial bundle $E=M\times\g$ has a Lie algebroid structure, with anchor given by the action map, and with the Lie bracket on sections extending the Lie bracket of $\g$ (regarded as constant sections of $M\times \g$).
\item For a principal $G$-bundle $P\to M$, the bundle $E=TP/G$ is a Lie algebroid, 
known as the \emph{Atiyah algebroid}. Its sections are identified with the $G$-invariant vector fields on $M$. 
\item Let $N\subset M$ be a codimension $1$ submanifold. Then there is 
a Lie algebroid $E$ whose sections are the vector fields on $M$ tangent 
to $N$, and another Lie algebroid $E'$ whose sections are the vector fields on $M$ that vanish along $N$. These Lie algebroids enter the Melrose \emph{b-calculus} \cite{mel:ati}; see also \cite[Section 17]{ca:ge}.
\end{itemize}

Given a vector bundle $V\to M$, let $\kappa_t\colon V\to V$ be scalar multiplication 
by $t\in\R$. For $t\neq 0$ this is a diffeomorphism. A multi-vector field 
$u\in \mf{X}^k(V)$ will be called (fiberwise) \emph{linear} if it is homogeneous of degree 
$1-k$, that is, 
\[  \kappa_t^* u=t^{1-k}\ u\]
for $t\neq 0$. In particular, a bivector field is linear if it is homogeneous of degree $-1$. 
 The following theorem gives a 1-1 correspondence 
\vskip.05in
\begin{equation}\label{eq:correspondence2}
\Big\{\begin{tabular}{c} Vector bundles with \\
linear Poisson structures\end{tabular}\Big\}
\stackrel{1-1}{\longleftrightarrow} 
\Big\{\mbox{ Lie algebroids}\Big\}
\end{equation}
\vskip.05in

For any section $\sigma\in \Gamma(E)$, let $\phi_\sigma\in C^\infty(E^*)$ be the corresponding linear function on the dual bundle $E^*$.

\begin{theorem}\label{th:la}
For any Lie algebroid $E$ over $M$, the total space of the dual bundle $E^*$ has a unique Poisson bracket such that for all sections $\sigma,\tau\in \Gamma(E)$, 
\begin{equation}\label{eq:liealgebroid}
\{\phi_\sigma,\phi_\tau\}=\phi_{[\sigma,\tau]}.\end{equation}
The Poisson structure is linear; conversely, every fiberwise linear Poisson structure on a vector bundle $V\to M$ arises in this way from a unique Lie algebroid structure on the dual bundle $V^*$.
\end{theorem}
\begin{proof}
Let $E\to M$ be a Lie algebroid. The differentials of the linear functions span the cotangent spaces to $E^*$ everywhere, except along the zero section $M\subset E^*$. 
Hence, there can be at most one  bivector field $\pi\in \mf{X}^2(E^*)$ such that 
the corresponding bracket $\{\cdot,\cdot\}$ satisfies \eqref{eq:liealgebroid}. To show its existence, 
we may argue with local bundle trivializations $E|_U=U\times \R^n$ over open subsets $U\subset M$. 
Let $\epsilon_1,\ldots,\epsilon_n$ be the corresponding basis of sections, define $c_{ij}^k\in C^\infty(U)$ by $[\epsilon_i,\epsilon_j]=\sum_k c_{ij}^k \epsilon_k$, and let 
$\a_i=\a(\epsilon_i)\in\mf{X}(U)$. 
Letting $y_i$ be the coordinates on 
$(\R^n)^*$ corresponding to the basis, one finds that
\begin{equation}\label{eq:local}
 \pi=\hh \sum_{ijk} c_{ij}^k\, y_k \f{\p}{\p y_i}\wedge \f{\p}{\p y_j}+\sum_i 
 \f{\p}{\p y_i}\wedge \a_i\end{equation}
is the unique bivector field on $E^*|_U=U\times (\R^n)^*$ satisfying \eqref{eq:liealgebroid}. (Evaluate the two sides on $\sigma=\sum_i f^i \eps_i$ and $\tau=\sum_i g^i \eps_i$.)  This proves the existence of $\pi\in \mf{X}^2(E^*)$. The Jacobi identity for $\{\cdot,\cdot\}$ holds true since it is satisfied on linear functions, by the Jacobi identity for $\Gamma(E)$. 

Conversely, suppose $p\colon V\to M$ is a vector bundle with a linear Poisson structure $\pi$. Let $E=V^*$ be the dual bundle. We define the Lie bracket on sections and the anchor $\a\colon E\to TM$ by 
\[ \phi_{[\sigma,\tau]}:=\{\phi_\sigma,\phi_\tau\},\ \ \ p^*\big(\a(\sigma)(f)\big):=
\{\phi_\sigma,\ p^*f\}\]
for $f\in C^\infty(M)$ and $\sigma,\tau\in \Gamma(E)$. This is well-defined: for instance, 
since $\phi_\sigma$ and $p^* f$ have homogeneity $1$ and $0$ respectively, their Poisson bracket is homogeneous of degree $1+0-1=0$. Also, it is straightforward to check that $\a(\sigma)$ is a vector field, and that the map $\sigma\mapsto \a(\sigma)$ is $C^\infty(M)$-linear. The Jacobi identity for the bracket $[\cdot,\cdot]$ follows from 
that of the Poisson bracket, while the Leibnitz rule \eqref{eq:leibnitz}
for the anchor $\a$ follows from the derivation property of the Poisson bracket, as follows:
\begin{align*}
\phi_{[\sigma,f\tau]}&=\{\phi_\sigma,\phi_{f\tau}\}\\
&=\{\phi_\sigma,(p^*f)\, \phi_\tau\}\\ 
&= p^*(\a(\sigma)f)\,\phi_\tau+(p^* f)\, \phi_{[\sigma,\tau]}.& & \qedhere
\end{align*}
\end{proof}

\subsection{Basic properties of Poisson structures}
Let $\pi\in \mf{X}^2(M)$ be a bivector field, with corresponding bundle map 
$\pi^\sharp\colon T^*M\to TM$. For $f\in C^\infty(M)$, let $X_f=\pi^\sharp(\d f)$ be the corresponding Hamiltonian vector field. Let $\{\cdot,\cdot\}$ be the bracket defined 
by $\{f,g\}=\pi(\d f,\d g)$. 
\begin{proposition} We have the equivalences, 
\begin{eqnarray*}
\mbox{$\{\cdot,\cdot\}$ is a Poisson bracket}
&\Leftrightarrow & [X_f,X_g]=X_{\{f,g\}}\ \ \mbox{ for all }f,g\\
&\Leftrightarrow & \L_{X_f}\pi=0\ \  \mbox{ for all } f\\
&\Leftrightarrow & \L_{X_f}\circ \pi^\sharp=\pi^\sharp \circ \L_{X_f} \ \ \mbox{ for all } f 
\end{eqnarray*} 
\end{proposition}
\begin{proof}
The third equivalence is clear; the others follow from alternative expressions for the Jacobiator \eqref{eq:jac} which we leave as exercises: 
\begin{align*}\on{Jac}(f,g,h) &= \L_{[X_f,X_g]}h-\L_{X_{\{f,g\}}}h \\
&= (\L_{X_f}\pi)(\d g,\d h).
\qedhere
\end{align*}
\end{proof}

\subsection{Symplectic foliation}
Let $(M,\pi)$ be a Poisson manifold. For any $m\in M$, the skew-symmetric bilinear form $\pi_m$ on $T^*_mM$ has kernel $\ker(\pi_m^\sharp)$; hence it descends to a non-degenerate bilinear form on the quotient space 
\[ T^*_mM/\ker(\pi_m^\sharp)\cong 
\big(\on{ann}(\ker(\pi_m^\sharp))\big)^*=
\on{ran}(\pi_m^\sharp)^*.\]
That is, the subspace \[ \mathsf{S}_m=\on{ran}(\pi_m^\sharp)\subseteq T_mM\] inherits a symplectic structure, or equivalently a non-degenerate Poisson tensor $\pi_{\mathsf{S}_m}\in \wedge^2 \mathsf{S}_m$ such that the inclusion map $\mathsf{S}_m\hra T_mM$ takes $\pi_{\mathsf{S}_m}$ to $\pi_m$.

The subset $\on{ran}(\pi^\sharp)\subset TM$ is usually a \emph{singular distribution}, since the dimensions of these subspaces need not be constant. Nevertheless, 
one can define its leaves: 
\begin{definition}
An connected immersed submanifold $j\colon S\to M$ is called a \emph{symplectic leaf} of the Poisson manifold $(M,\pi)$ if for all $s\in S$, 
\[ (T_s j)(T_sS)=\on{ran}(\pi^\sharp_{j(s)}).\]
\end{definition}
For a symplectic leaf, we obtain a bivector field $\pi_S\in\mf{X}^2(S)$ such that
$\pi_S\sim_j \pi$. The corresponding bracket is characterized by 
$\{j^*f,j^*g\}_S=j^*\{f,g\}$; since $j$ is an immersion it follows that $\pi_S$ is a Poisson bracket.  Since $\on{ran}(\pi_S^\sharp)=TS$ by construction, it defines a symplectic 2-form $\omega_S\in \Om^2(S)$ such that 
$\omega_S^\flat=-(\pi_S^\sharp)^{-1}$. In Section \ref{sec:weinsteinsplitting} we 
will prove the following fundamental  result:
 \vskip.1in
 \begin{theorem} \cite{wei:loc} \label{th:leaves}
 Every point $m$ of a Poisson manifold $M$ is contained in a unique maximal symplectic leaf $S$. 
 \end{theorem}
\vskip.1in
Thus, $M$ has a decomposition into maximal symplectic leaves. For regular Poisson structures (i.e., such that $\pi^\sharp$ has constant rank), the integrability of $\on{ran}(\pi^\sharp)$ follows from Frobenius' theorem, since  $[X_f,X_g]=X_{\{f,g\}}$. In the general case, one can still obtain the leaf through a given point $m$ as the `flow-out' of $m$ under all Hamiltonian vector fields, and this is the argument used in \cite{wei:loc}. In Section \ref{sec:weinsteinsplitting}, we will use a different approach, and deduce Theorem \ref{th:leaves} as a corollary to the \emph{Weinstein splitting theorem} for Poisson structures. 
\begin{example}
For $M=\g^*$ the dual of a Lie algebra $\g$, the symplectic leaves are the orbits of coadjoint action $G$ on $\g^*$. Here $G$ is any connected Lie group integrating $\g$.  
\end{example}
\begin{example}
For a Poisson structure $\pi$ on a 2-dimensional manifold $M$, let $Z\subset M$ be its set of zeros, i.e. points $m\in M$ where $\pi_m=0$. Then the 2-dimensional symplectic leaves of $\pi$ are the connected components of $M-Z$, while the $0$-dimensional leaves are the points of $Z$. 
\end{example}

\subsection{Notes}
While Poisson brackets  have long been used as a formalism in physics, going back to the original work of Poisson, the study of Poisson structures on manifolds was formally launched only in 1979 through the work of Lichnerowicz \cite{lic:var}. Many of the basic results about Poisson manifolds were proved in Weinstein's classical paper \cite{wei:loc}. The Poisson structure on the dual of a Lie algebra already appears in Lie's work; it was rediscovered by Berezin \cite{ber:rem}, and played in important role in Kirillov-Kostant-Souriau's \emph{orbit method} in representation theory.  Lie algebroids were introduced by Pradines \cite{pra:th} as the infinitesimal counterparts to Ehresmann's notion of Lie groupoids; see Appendix \ref{app:groupoids}.
For more information on the history of Poisson geometry, see the book \cite{kos:his} by Kosmann-Schwarzbach.

 \section{Dirac manifolds}\label{sec:diracmanifolds}
\subsection{Dirac structures} 
It turns out to be extremely useful to view Poisson structures $\pi$ in terms of the 
graph $\on{Gr}(\pi)$ of the bundle map $\pi^\sharp\colon T^*M\to TM,\ \mu\mapsto \pi(\mu,\cdot)$. 
Let 
\begin{equation}\label{eq:double}
 \T M=TM\oplus T^*M
 \end{equation}
be the direct sum of the tangent and cotangent bundles. Elements of this bundle will be written $x=v+\mu$, with $v\in T_mM$ and $\mu\in T^*_mM$, and similarly 
sections as $\sig=X+\alpha$, where $X$ is a vector field and $\alpha$ a 1-form. We will denote by 
\begin{equation}\label{eq:anchor}
 \a\colon \T M\to TM\end{equation} 
the projection to the first summand; thus 
$\a(v+\mu)=v$. 
For any subbundle $E\subset \T M$, we denote by $E^\perp$ its orthogonal with respect to the symmetric bilinear form 
(of split signature), 
\begin{equation}\label{eq:bilinearform}
\l v_1+\mu_1,\ v_2+\mu_2\r=\l \mu_1,v_2\r+\l\mu_2,v_1\r;
\end{equation}
here $v_1,v_2\in TM$ and $\mu_1,\mu_2\in T^*M$ (all with the same base point in $M$). 
The subbundle $E$ is called maximal isotropic, or \emph{Lagrangian}  if $E=E^\perp$. 
(The terminology is borrowed from symplectic geometry, where it is used for maximal isotropic subspaces for non-degenerate \emph{skew-symmetric} forms.)
Given a bivector field $\pi \in \mf{X}^2(M)$, its \emph{graph}
\[ E=\on{Gr}(\pi)\subset \T M,\]
given as the set of all  $\pi^\sharp(\mu)+\mu$ for $\mu\in T^*M$, is Lagrangian; 
in fact the Lagrangian subbundles $E\subset \T M$ with $E\cap TM=0$ are exactly the graphs of bivector fields. 

To formulate the integrability condition for $\pi$, we need the 
\emph{Courant bracket} (also know as the \emph{Dorfman bracket}) 
\begin{equation}\label{eq:courantbracket}
 \Cour{\sigma_1,\sigma_2}=[X_1,X_2]+\L_{X_1}\alpha_2-\iota_{X_2}\d\alpha_1\end{equation}
on sections $\sigma_i=X_i+\alpha_i\in\Gamma(\T M)$.  It is straightforward to check that this bracket has the following properties,  for all sections: 
\begin{align}
\a(\sigma_1)\l \sigma_2,\sigma_3\r
&=\l \Cour{\sigma_1,\sigma_2},\sigma_3\r
+\l \sigma_2,\ \Cour{\sigma_1,\sigma_3}\r,\label{eq:i}\\
  \Cour{\sigma_1,\Cour{\sigma_2,\sigma_3}}
&=\Cour{\Cour{\sigma_1,\sigma_2},\sigma_3}+
\Cour{\sigma_2,\Cour{\sigma_1,\sigma_3}},\label{eq:ii}\\
\Cour{\sigma,\tau}+\Cour{\tau,\sigma}&=\d\, \l\sigma,\tau\r.\label{eq:iii}
\end{align}
%
Note that the Courant bracket is not skew-symmetric. However, it restricts to a skew-symmetric bracket on sections of Lagrangian subbundles since the right hand side of
\eqref{eq:iii} is zero on such sections. One also has the Leibnitz identity
\begin{equation}\label{eq:leibnitzc}
\Cour{\sigma,f\tau}=f\Cour{\sigma,\tau}+\ca{L}_{\a(\sigma)}(f)\ \tau
\end{equation}
for $\sigma,\tau\in \Gamma(\T M)$ and $f\in C^\infty(M)$. 
\begin{definition}\cite{cou:di,couwein:beyond}
A \emph{Dirac structure} on $M$ is a Lagrangian subbundle $E\subset \T M$ whose  space of sections is closed under the Courant bracket.
\end{definition}
\begin{proposition}\label{prop:diraclie}
Any Dirac structure $E\subset \T M$ acquires the structure of a Lie algebroid, with the Lie bracket on sections given by 
the Courant bracket on $\Gamma(E)\subset \Gamma(\T M)$, and with the anchor obtained by restriction of the anchor $\a\colon \T M\to TM$. 
 \end{proposition}
\begin{proof}
By \eqref{eq:iii}, the Courant bracket is skew-symmetric on sections of $E$, and \eqref{eq:ii} gives the Jacobi identity.
The Leibnitz identity follows from that for the Courant bracket, Equation \eqref{eq:leibnitzc}.
\end{proof}

The integrability of a Lagrangian subbundle $E\subset \T M$ is equivalent to the vanishing of the expression 
\begin{equation}\label{eq:couranttensor}
 \Upsilon_E(\sigma_1,\sigma_2,\sigma_3)=\l \sigma_1,\ \Cour{\sigma_2,\sigma_3}\r\end{equation}
for all $\sigma_1,\sigma_2,\sigma_3\in\Gamma(E)$. Indeed, given $\sigma_2,\sigma_3\in\Gamma(E)$, the vanishing for all $\sigma_1\in \Gamma(E)$ means precisely that 
$\Cour{\sigma_2,\sigma_3}$ takes values in $E^\perp=E$. Using the properties \eqref{eq:i} and \eqref{eq:iii} of the Courant bracket, 
one sees that $\Upsilon_E$ is skew-symmetric in its entries. Since $\Upsilon_E$ is clearly  tensorial in its first entry, it follows that 
it is tensorial in all three entries: that is
\[ \Upsilon_E\in \Gamma(\wedge^3 E^*).\]
In particular, to calculate $\Upsilon_E$ it suffices to determine its values on any collection of sections that span $E$ everywhere.

\subsection{Poisson structures as Dirac structures}
\begin{proposition}
A bivector field $\pi\in\mf{X}^2(M)$ is Poisson if and only if its graph 
$\on{Gr}(\pi)$  is a Dirac structure. 
\end{proposition}
\begin{proof}
We will show that $\Upsilon_{\on{Gr}(\pi)}$ 
coincides with $\Upsilon_\pi$ (cf. Section \ref{subsec:poisson}), under the isomorphism $\on{Gr}(\pi)\cong T^*M$ given by projection along $TM$.  It suffices to consider sections of the form $X_f+\d f$ for $f\in C^\infty(M)$, where $X_f=\pi^\sharp(\d f)$. 
Thus let $\sigma_i=X_{f_i}+\d f_i$. We have 
\[ \Cour{\sigma_2,\sigma_3}=[X_{f_2},X_{f_3}]+\d\, \ca{L}_{X_{f_2}}(f_3),\]
hence 
\[ \l\sigma_1,\Cour{\sigma_2,\sigma_3}\r=\ca{L}_{[X_{f_2},X_{f_3}]}(f_1)+
\ca{L}_{X_{f_1}}\ca{L}_{X_{f_2}}(f_3)=\on{Jac}(f_1,f_2,f_3).\]
The result follows. 
\end{proof}
Combining with Proposition \ref{prop:diraclie}, it follows that the graph of any Poisson structure $\pi$ is a Lie algebroid. Using the 
bundle isomorphism $\on{Gr}(\pi)\cong T^*M$ (forgetting the vector field part) we  obtain:
\begin{corollary}
For any Poisson manifold $(M,\pi)$, the cotangent bundle $T^*M$ is a Lie algebroid, with bracket on 
$\Gamma(T^*M)=\Om^1(M)$ given by 
\begin{equation}\label{eq:1formbracket}
[\alpha,\beta]=\L_{\pi^\sharp(\alpha)}\beta-\iota_{\pi^\sharp(\beta)}\d\alpha\end{equation}
and with anchor $\pi^\sharp$. 
\end{corollary}
\begin{remark}
Note that on exact 1-forms, the bracket \eqref{eq:1formbracket} simplifies to $ [\d f,\d g]=\d\{f,g\}$. 
That is, $\d\colon C^\infty(M)\to \Omega^1(M)$ is a Lie algebra morphism. 
\end{remark}
\begin{remark}
The condition that \eqref{eq:1formbracket}
be a Lie bracket is another equivalent characterization of Poisson structures. 
\end{remark}
\begin{remark} Recall that a Lie algebroid structure on a vector bundle is equivalent to a 
linear Poisson structure on the total space of the dual bundle (Theorem \ref{th:la}). 
Applying this to the cotangent  Lie algebroid
of a Poisson manifold $(M,\pi)$, it follows that the total space of the tangent bundle 
$TM=(T^*M)^*$ inherits a Poisson structure. This Poisson structure admits an alternative description as the \emph{tangent lift} of the Poisson structure $\pi$. 
\end{remark}

\subsection{Gauge transformations of Poisson and Dirac structures}
Any diffeomorphism $\Phi\colon M\to M$ defines an automorphism 
\[ \T\Phi\colon \T M\to \T M\]
by taking the sum of the tangent and cotangent maps. This  automorphism preserves the Courant bracket,  and in particular takes Dirac structures  $E\subset \T M$ to Dirac structures $(\T \Phi)(E)$. For Poisson structures $\pi$, 
\[ (\T\Phi)(\on{Gr}(\pi))=\on{Gr}(\Phi_* \pi),\]
the graph of the Poisson structure $\Phi_*\pi=(T\Phi)(\pi)$.  
More interestingly, every closed 2-form $\omega\in \Om^2_{cl}(M)$ also defines an automorphism 
\[ \ca{R}_\omega\colon \T M\to \T M,\ 
x\mapsto x+\iota_{\a(x)}\omega.\]
%
It is an exercise to check that $\ca{R}_\omega$ preserves the metric and the Courant bracket. Again, $\ca{R}_\omega$ acts on the set of Dirac structures. If $\pi$ is a Poisson structure, and if $\ca{R}_\omega(\on{Gr}(\pi))$ is transverse to $TM$ (e.g., if $\omega$ is `sufficiently small' relative to $\pi$), then it defines a new Poisson structure $\pi^\omega$ by 
\[ \ca{R}_\omega (\on{Gr}(\pi))=\on{Gr}(\pi^\omega).\]
\begin{lemma}\cite{sev:poi1}
The Poisson structure $\pi^\omega$ satisfies $\on{ran}((\pi^\omega)^\sharp)=\on{ran}(\pi^\sharp)$. For $m\in M$,  the symplectic 2-forms on $\mathsf{S}_m=\on{ran}(\pi_m^\sharp)$ corresponding to $\pi^\omega$ and $\pi$ differ by the restriction of $-\omega$ to $\mathsf{S}_m$.
\end{lemma}
\begin{proof}
The first claim $\on{ran}((\pi^\omega)^\sharp)=\on{ran}(\pi^\sharp)$ follows from 
\[ \a\big(\on{Gr}(\pi^\omega)\big)=
\a\big(\ca{R}_\omega(\on{Gr}(\pi)\big)=\a\big(\on{Gr}(\pi)\big).\]
For the second claim, given $m\in M$, let $\sig_m$ be the 2-form on $\mathsf{S}_m=\on{ran}(\pi_m^\sharp)$ defined by $\pi_m$. If $v_1,v_2\in \mathsf{S}_m$
with $v_i=\pi_m^\sharp(\mu_i)$, we have 
\[ \sig_m(v_1,v_2)=-\pi_m(\mu_1,\mu_2)=\l\mu_1,v_2\r.\]
The gauge transformation takes $v_i+\mu_i$ to $v_i+\mu_i+\iota_{v_i}\omega$, 
hence the right hand side changes to 
\[\l \mu_1+\iota_{v_1}\omega,v_2\r=\sigma_m(v_1,v_2)+\omega_m(v_1,v_2).\qedhere\]  
\end{proof}
One calls $\pi^\omega$ the \emph{gauge transformation} of $\pi$ by the 
closed 2-form $\omega$. More generally, for a Dirac structure $E\subset \T M$ one calls $\ca{R}_\omega(E)$ the gauge transformation of the Dirac structure.
\begin{proposition}[Automorphisms of the Courant bracket]\label{prop:couraut}
\cite{gua:ge1}
\begin{enumerate}
\item 
The group of Courant automorphisms of $\T M$ (vector bundle automorphisms preserving the metric and the bracket, and compatible with the anchor) is a semi-direct product 
\[ \on{Aut}_{CA}(\T M)=\Om^2_{cl}(M)\rtimes \on{Diff}(M),\]
where $(\omega,\Phi)$ acts as $ \ca{R}_{-\omega}\circ \T \Phi$. 
\item The Lie algebra  of infinitesimal Courant automorphisms is a semi-direct product 
\[\mf{aut}_{CA}(\T M)= \Om^2_{cl}(M)\rtimes \mf{X}(M),\]
where the action of $(\gamma,X)$ on a section $\tau=Y+\beta \in \Gamma(\T M)$
is given by 
\begin{equation}\label{eq:infaction}
 (\gamma,X).\tau=[X,Y] +\L_X\beta - \iota_Y \gamma.\end{equation}
\item For any section $\sigma=X+\alpha\in \Gamma(\T M)$, 
the Courant bracket $\Cour{\sigma,\cdot}$ is the infinitesimal automorphism 
$(\d\alpha,\,X)$. 
\end{enumerate}
\end{proposition}
\begin{proof}
a) Given $A\in \on{Aut}_{CA}(\T M)$, let $\Phi\in \on{Diff}(M)$ be the base map. 
Then $A'=A\circ \T \Phi^{-1}$ has base map the identity map. In particular, 
$\a\circ A'=\a$, i.e. for all $v\in TM\subset \T M$, $A'v-v\in T^*M$. Let 
$\omega(v,w)=\l A'(v),w\r$. Since $A'$ preserves the metric, 
\[ 0=\l A'(v),A'(w)\r=\l A'(v),w\r+\l v,A'(w)\r
=\omega(v,w)+\omega(w,v),\]
thus $\omega$ is skew-symmetric. It follows that $A'=\ca{R}_\omega$. 
But $\ca{R}_\omega$ preserves the Courant bracket if and only if $\omega$ is closed. 
This proves $A=\ca{R}_\omega\circ \T \Phi$. The proof of b) is similar, while c) 
is immediate from \eqref{eq:infaction} and the formula for the Courant bracket. 
\end{proof}

We are interested in the integration of infinitesimal Courant automorphism, especially those generated by sections of $\T M$. In the discussion below, we will be vague about issues of completeness of vector fields; in the general case one has to work with \emph{local} flows.
The following result is an infinite-dimensional instance 
of a formula for time dependent flows on semi-direct products $V\rtimes G$, where $G$ is a Lie group and $V$ a $G$-representation.

\begin{proposition}\label{prop:integrate} \cite{gua:ge1,hu:ham}
Let $(\om_t,\Phi_t)\in \Aut_{CA}(\T M)$ be the family of automorphisms integrating the time-dependent infinitesimal automorphisms 
$(\gamma_t,X_t)\in\mf{aut}_{CA}(\T M)$. Then $\Phi_t$ is the flow of $X_t$, while 
\[ \omega_t=\int_0^t \big((\Phi_s)_* \gamma_s\big) \d s.\]
\end{proposition}
\begin{proof}  Recall (cf. Appendix \ref{app:signs}) that 
the flow $\Phi_t$ of a time dependent vector field $X_t$ is defined in terms of the action on functions by $\f{d}{ d t}(\Phi_t)_*=(\Phi_t)_*\circ \L_{X_t}$. Similarly, the 1-parameter family of Courant automorphisms $(\om_t,\Phi_t)$ integrating $(\gamma_t,X_t)$ is defined in terms of the action on sections $\tau\in \Gamma(\T M)$ by 
\[ \f{d}{ d t}\Big( (\om_t,\Phi_t).\tau\Big)=
(\om_t,\Phi_t).(\gamma_t,X_t).\tau.\] 
Write $\tau=Y+\beta\in \Gamma(\T M)$. Then 
\begin{align*}
 \f{d}{ d t}\Big( (\om_t,\Phi_t).\tau\Big)&= \f{d}{ d t}\Big( 
 (\Phi_t)_*\tau-\iota\big((\Phi_t)_*Y\big)\omega_t
 \Big)\\
 &= (\Phi_t)_* \L_{X_t} \tau-
\iota\big((\Phi_t)_*\L_{X_t} Y\big)\omega_t
-\iota\big((\Phi_t)_*Y\big)\f{ d \omega_t}{d t}.
\end{align*}
On the other hand, 
\begin{align*}
(\om_t,\Phi_t).(\gamma_t,X_t).\tau
&=(\om_t,\Phi_t).\Big(\L_{X_t}\tau-\iota(Y)\gamma_t\Big)
\\
&=(\Phi_t)_*\L_{X_t}\tau-\iota\big((\Phi_t)_* Y\big)(\Phi_t)_*\gamma_t
-\iota\big((\Phi_t)_* \L_{X_t} Y\big)\omega_t.
\end{align*}
Comparing, we see $(\Phi_t)_*\gamma_t=\f{ d }{d t}\omega_t$. 
\end{proof}

This calculation applies in particular to the infinitesimal automorphisms $(\gamma_t,X_t)$ defined by $\sigma_t=X_t+\alpha_t\in\Gamma(\T M)$; here $\gamma_t=\d\alpha_t$. Note that in this case,
\[ \omega_t=\d\,\int_0^t \big((\Phi_s)_* \alpha_s\big) \d s\]
is a family of \emph{exact} 2-forms.

%
\subsection{Moser method for Poisson manifolds}
The standard \emph{Moser argument} for symplectic manifolds 
shows that for a compact symplectic manifold, any 1-parameter family of deformations of the symplectic forms in a prescribed cohomology class is obtained by the action of a 1-parameter family of diffeomorphisms. The following version for Poisson manifolds can be proved from the symplectic case, arguing `leaf-wise', or more directly using the Dirac geometric methods described above. 
\begin{theorem}\label{th:moser}\cite{al:gw,al:lin}
Suppose $\pi_t\in\mf{X}^2(M)$ is a 1-parameter family of Poisson structures related by gauge transformations, 
\[ \pi_t=(\pi_0)^{\omega_t},\]
where $\omega_t\in\Om^2(M)$ is a family of closed 2-forms. Suppose that 
\[ \f{ d \omega_t}{d t}=-\d a_t,\]
with a smooth family of 1-forms $a_t\in \Om^1(M)$, defining a time dependent vector field $X_t=\pi_t^\sharp(a_t)$. Let $\Phi_t$ be the flow of $X_t$. Then
\[ (\Phi_t)_*\pi_t=\pi_0.\]
\end{theorem}
\begin{proof}
By definition, 
\[ X_t+a_t\in \Gamma\big(\on{Gr}(\pi_t)\big)=\Gamma\Big(\ca{R}_{\omega_t}\big(\on{Gr}(\pi_0)\big)\Big).\]
Applying the transformation $\ca{R}_{-\omega_t}$, we obtain 
\[  X_t+b_t \in 
\ca{R}_{-\omega_t}\big(\on{Gr}(\pi_t)\big)=
\Gamma\big(\on{Gr}(\pi_0)\big)\]
with  $b_t=a_t-\iota(X_t)\omega_t$. Since this is a section of the Dirac structure 
$\on{Gr}(\pi_0)$, the  flow generated by $(\d  b_t,\ X_t)$ preserves $\on{Gr}(\pi_0)$. 
Using Proposition \ref{prop:integrate} and the calculation 
\begin{equation}\label{eq:needthis}
 \f{ d }{d t}\Big((\Phi_t)_* \omega_t\Big)=(\Phi_t)_* \Big(
\f{ d \omega_t}{d t}+\L_{X_t}\omega_t\Big)=-\d \big((\Phi_t)_* b_t\big),
\end{equation}
we see that this flow is 
given by $\big(\!-(\Phi_t)_* \omega_t,\Phi_t\big)$. Hence 
\begin{align*}
\on{Gr}(\pi_0)&=(-(\Phi_t)_*\omega_t,\Phi_t).\on{Gr}(\pi_0)\\
&= \ca{R}_{(\Phi_t)_*\omega_t} \circ \T \Phi_t (\on{Gr}(\pi_0))\\
&=\T \Phi_t \circ \ca{R}_{\omega_t} (\on{Gr}(\pi_0))\\
&=\T \Phi_t (\on{Gr}(\pi_t))\\ 
&=\on{Gr}\big((\Phi_t)_*\pi_t\big). &&\qedhere
\end{align*}
\end{proof}

\subsection{Pull-backs}\label{subsec:pullbackdirac}
In general, there is no natural way of pulling back a Poisson structure under a 
smooth map $\varphi\colon N\to M$. However, under a transversality assumption it can always be pulled back \emph{as a Dirac structure}. 
\begin{definition}
Let $\varphi\colon N\to M$ be  a smooth map.
\begin{enumerate}
\item For $x=v+\mu\in \T_m M$, $y=w+\nu\in \T_n N$, we write 
\[ y\sim_\varphi x\]
if $m=\varphi(n)$ and $v=(T_n\varphi)\,w,\ \nu=(T_n\varphi)^*\mu$. 
\item
For sections $\sigma=X+\alpha\in \Gamma(\T M)$ and $\tau=Y+\beta\in \Gamma(\T N)$  we write 
\[ \tau\sim_\varphi \sigma\] 
if $Y\sim_\varphi X$ (related vector fields\footnote{Recall that vector fields $Y\in\mf{X}(N)$ and $X\in\mf{X}(M)$ are \emph{$\varphi$-related} (written $Y\sim_\varphi X$) if 
$(T_n\varphi)(Y_n)=X_{\varphi(n)}$ for all $n\in N$. If $Y_1\sim_\varphi X_1$ and $Y_2\sim_\varphi X_2$ then $[Y_1,Y_2]\sim_\varphi [X_1,X_2]$.}) and $\beta=\varphi^*\alpha$. 
\item
For a subbundle $E\subset \T M$, we write 
\[ \varphi^!E=\{y\in \T N|\ \exists x\in E\colon y\sim_\varphi x\}\]
\end{enumerate}
\end{definition}
Observe that the relation preserves metrics, in the sense that 
\begin{equation}\label{eq:innerproducts}
 y_1 \sim_\varphi x_1,\ \ y_2 \sim_\varphi x_2\ \ \ \Rightarrow\ \ \ 
 \l y_1,y_2\r =\l x_1,x_2\r.
\end{equation}
Note also that if $\varphi$ is a diffeomorphism, then $y\sim_\varphi x$ holds if and only if 
$x=(\T \varphi) (y)$. 
In general, one needs transversality assumptions to ensure that 
$\varphi^!E$ is again subbundle. We will consider the following situation: 
\begin{proposition}
Suppose $E\subset \T M$ is a Dirac structure, and $\varphi\colon N\to M$ is transverse to the anchor of $E$. Then $\varphi^!E$ is again a Dirac structure. 
\end{proposition}
\begin{proof}
Consider first the case that $\varphi$ is the embedding of a submanifold, $\varphi\colon N\hra M$. 
The transversality condition ensures that $\varphi^!E$ is a subbundle of $\T N$ of the right dimension; the property \eqref{eq:innerproducts} shows that it is Lagrangian. It also 
follows from the transversality that for any $y\in \varphi^!E$, the element $x\in E$ such that $y\sim_\varphi x$ is \emph{unique}; this defines an inclusion $\varphi^!E\hra E|_N$
with image $\a^{-1}(TN)\cap E$.  
Hence, any section $\tau\in\Gamma(\varphi^!E)$  extends (non-uniquely) to a section 
$\sigma\in \Gamma(E)$ such that $\tau\sim_\varphi \sigma$;  conversely, given $\sigma\in \Gamma(E)$ such that $\a(\sigma)$ is tangent to $N$ 
we have $\tau\sim_\varphi \sigma$ for a (unique) section $\tau$. Suppose $\tau_1,\tau_2$ are sections of $\varphi^!E$, and choose $\sigma_i\in \Gamma(E)$ such that $\tau_i\sim_\varphi \sigma_i$. By definition of the Courant bracket, $\Cour{\tau_1,\tau_2}\sim_\varphi \Cour{\sigma_1,\sigma_2}$. Hence $\Cour{\tau_1,\tau_2}$ is a section of 
$\varphi^!E$.  This proves the proposition for the case of an embedding. In the general case, 
given $\varphi$ consider the embedding of $N$ as the graph of $\varphi$, 
\[ j\colon N\to M\times N,\ 
n\mapsto (\varphi(n),n).\] 
It is easy to see that 
\[ \varphi^!E=j^!(E\times TN)\]
as subsets of $\T N$. Since $j^!(E\times TN)$ is a Dirac structure by the above, we are done. 
\end{proof}

In particular, if $\pi$ is a Poisson structure on $M$, and $\varphi\colon N\to M$ is transverse to the map $\pi^\sharp$, we can define the pull-back $\varphi^!\on{Gr}(\pi)\subset \T N$ \emph{as a Dirac structure}. In general, this is not a Poisson structure. 
We have the following necessary and sufficient condition. 
\begin{proposition}\label{prop:cosymplectic}
Suppose $(M,\pi)$ is a Poisson manifold, and $\varphi\colon N\to M$ is transverse to $\pi^\sharp$. Then $\varphi^!\on{Gr}(\pi)\subset \T N$ defines a Poisson structure 
$\pi_N$ if and only if $\varphi$ is an immersion, with 
\begin{equation}\label{eq:cosymplectic}
 TM|_N=TN\oplus \pi^\sharp(\on{ann}(TN)).
\end{equation}
\end{proposition}
\begin{proof}
$\varphi^!\on{Gr}(\pi)$ defines a Poisson structure if and only if it is transverse to 
$TN\subset \T N$. But $\varphi^!\on{Gr}(\pi)\cap TN$ contains in particular 
elements $y\in \T N$ with $y\sim_\varphi 0$. Writing $y=w+\nu$ with a tangent vector $w$ and covector $\nu$, this means that $\nu=0$ and 
$w\in \ker(T\varphi)$. Hence, it is necessary that $\ker(T\varphi)=0$. Let us therefore assume that  $\varphi$ is an immersion. For $w\in TN\subset \T N$, we have  
\[ w\sim_\varphi \pi^\sharp(\mu) +\mu\in \on{Gr}(\pi)\ \ \ \Leftrightarrow \ \ \ 
(T\varphi)(w)=\pi^\sharp(\mu),\ \ \mu\in \ker(\varphi^*)=\on{ann}(TN).\]
Hence, the condition $\varphi^!E\cap TN=0$ is equivalent to $\pi^\sharp(\on{ann}(TN))\cap TN=0$. 
\end{proof}
\begin{definition}
A submanifold $N\subset M$ of a Poisson manifold $(M,\pi)$ with the property \eqref{eq:cosymplectic} is called a \emph{cosymplectic submanifold}. 
\end{definition}
This type of submanifold was first considered in \cite[Proposition 1.4]{wei:loc}. One consequence 
of \eqref{eq:cosymplectic} is that the bivector field $\pi$ restricts to a non-degenerate 
bilinear form on $\on{ann}(TN)$. That is, 
\begin{equation}\label{eq:sympvb} \mathsf{V}=\pi^\sharp(\on{ann}(TN))\subset TM|_N\end{equation}
is a symplectic vector bundle $p\colon \mathsf{V}\to N$, with a fiberwise symplectic 2-form.

\begin{remark}
Note that cosymplectic submanifolds $N\subset M$ are necessarily transverse to the symplectic foliation. For this reason, they are also known as \emph{Poisson transversals} \cite{fre:nor}. In fact, the cosymplectic submanifolds are exactly those transverse submanifolds whose intersection with every symplectic leaf is a symplectic submanifold; the collection of such intersections is the symplectic foliation of $N$. 
\end{remark}

\subsection{Notes}
The idea of unifying Poisson geometry with pre-symplectic geometry, by viewing Poisson tensors and 2-forms through their graphs,  was introduced by Courant and Weinstein \cite{couwein:beyond}, and developed in Courant's thesis \cite{cou:di}. The name \emph{Dirac geometry} has to do with its relation to Dirac's theory of constrained mechanical systems. The non-skew-symmetric version of  the Courant bracket is due to Dorfman \cite{dor:dir}. (The original skew-symmetric version is rarely used nowadays.) 
Many of the later developments in Dirac geometry were inspired by unpublished notes of 
\v{S}evera \cite{sev:let}. Gauge transformations of Poisson structures were introduced by \v{S}evera-Weinstein \cite{sev:poi1} and further developed by Bursztyn \cite{bu:ga}. 
For pullbacks of Dirac structures, see e.g. \cite{bur:gauge}. The Courant bracket on $TM\oplus T^*M$ has been generalized to \emph{Courant algebroids} in the work of Liu-Xu-Weinstein \cite{liu:ma}; see e.g. \cite{lib:cou} for a more recent discussion. The Lie bracket on closed 1-forms on a Poisson manifold was first discovered by Fuchssteiner \cite{fuc:lie}, see also Koszul \cite{kos:cro}. 
Cosymplectic submanifolds are already discussed in Weinstein's paper \cite{wei:loc}, and later by Xu \cite{xu:dir},  Cattaneo-Zambon \cite{cat:coi}, and Frejlich-M\u{a}rcu\c{t} \cite{fre:nor}, among others. In his paper \cite{xu:dir}, Ping Xu discusses a
 more general (and very natural) class of submanifolds of Poisson manifolds, called \emph{Dirac submanifolds}, which also inherit Poisson structures. An even larger class of \emph{Poisson-Dirac submanifolds} is considered by Crainic-Fernandes in 
\cite{cra:intpoi}.

\section{Weinstein splitting theorem}\label{sec:weinsteinsplitting}
\subsection{The splitting theorem}
We begin with a statement of the theorem and some of its consequences. 
Let $(M,\pi)$ be a Poisson manifold, and $m\in M$. Then  
\[ \mathsf{S}:=\pi^\sharp(T^*_mM)\] 
is a symplectic vector space. Let $N\subset M$ be a submanifold through $m$, with the property that  
\[ T_mM= T_mN\oplus\mathsf{S}.\]
Then 
$\pi^\sharp(\on{ann}(TN))$ coincides with $\mathsf{S}$ at $m$, hence 
it is a complement to $TN$ at $m$. By continuity this also holds on some neighborhood of 
$m$. Thus, taking $N$ smaller if necessary, it is a cosymplectic submanifold, and hence inherits a Poisson structure $\pi_N$ (see Proposition \ref{prop:cosymplectic}). It is referred to as the \emph{transverse Poisson structure}. 

\begin{theorem}[Weinstein splitting theorem \cite{wei:loc}] The Poisson manifold 
$(M,\pi)$ is Poisson diffeomorphic near $m\in M$ to 
the product of Poisson manifolds,  
\[ N\times \mathsf{S}\]
where $\mathsf{S}$ is the symplectic vector space $\on{ran}(\pi_m^\sharp)$, and 
$N$ is a transverse submanifold as above, equipped with the transverse Poisson structure. 
More precisely, there exists a Poisson diffeomorphism 
between open neighborhoods of $m$ in $M$ and of $(m,0)$ in  $N\times \mathsf{S}$, 
taking $m$ to $(m,0)$, and with differential at $m$ equal to the given decomposition $T_mM\to T_m N\oplus \mathsf{S}$. 
\end{theorem}
\begin{remark}
Put differently, one can introduce local coordinates near $m$, in which $\pi$ decomposes as a sum of the standard  Poisson structure on $\R^{2k}$ with 
another Poisson structure on $\R^l$ having a zero at the origin. The remarkable fact is that one can eliminate  any `cross-terms'. 
\end{remark}
An immediate consequence of Weinstein's splitting theorem is the existence of a (local) \emph{symplectic foliation}: In the model $N\times \mathsf{S}$, it is immediate that $\{m\}\times \mathsf{S}$ is a symplectic leaf. But it also gives information on how the symplectic leaves fit together: The symplectic foliation of $M$ is locally a product of the foliation of the
lower-dimensional transversal $N$ with a symplectic vector space. Of course, the transverse Poisson structure on $N$ can still be quite complicated.

Weinstein's theorem has been generalized by Frejlich-M\u{a}rcu\c{t} \cite{fre:nor} to a normal form theorem 
around  arbitrary cosymplectic submanifolds $N\subset M$. Consider the symplectic vector bundle 
\[ p\colon \mathsf{V}=\pi^\sharp(\on{ann}(TN))\to N\] from \eqref{eq:sympvb}. 
As for any symplectic vector bundle, it is possible to find a closed 2-form $\omega$ on the total space of $\mathsf{V}$, such that $\omega(v,\cdot)=0$ for $v\in TN$, and such that $\omega$ pulls back to the given symplectic form on the fibers. (A particularly nice way of obtaining such a 2-form 
is the `minimal coupling' construction of Sternberg \cite{ste:min} and 
Weinstein \cite{wei:uni}.)  Consider the pull-back Dirac structure 
$ p^!\on{Gr}(\pi_N)$, where $p\colon \mathsf{V}\to N$ is the projection. Its gauge transformation by $\omega$, 
\begin{equation}\label{eq:VDirac}
\ca{R}_\omega\big(p^!\on{Gr}(\pi_N)\big)\subset \T \mathsf{V}\end{equation}
is transverse to $T\mathsf{V}$ near $N\subseteq \mathsf{V}$, hence it defines a Poisson structure $\pi_{\mathsf{V}}$ on a neighborhood of $N$, which furthermore agrees with $\pi$ 
along $N\subset \mathsf{V}$. The Poisson structures for  different choices of $\omega$ are related by the Moser method (Theorem \ref{th:moser}). 
\begin{theorem}[Frejlich-M\u{a}rcu\c{t} \cite{fre:nor}] \label{th:fre} Let $N\subset M$ be a cosymplectic submanifold of a Poisson manifold, with normal bundle $\mathsf{V}$. Define a Poisson structure $\pi_{\mathsf{V}}$ on a neighborhood of $N$ in $\mathsf{V}$, as explained above. Then there exists a tubular neighborhood embedding 
$\mathsf{V}\hra M$ which is a Poisson map on a possibly smaller neighborhood of $N\subset \mathsf{V}$. 
\end{theorem}
If the bundle $\mathsf{V}$ admits a trivialization $\mathsf{V}=N\times\mathsf{S}$, 
then the 2-form $\omega$ can simply be taken as pull-back of $\omega_{\mathsf{S}}$ under projection to the second factor. Furthermore, 
$p^!\on{Gr}(\pi_N)=\on{Gr}(\pi_N)\times T\mathsf{S}$ in this case, with 
\[  \ca{R}_\omega(p^!\on{Gr}(\pi_N))=\on{Gr}(\pi_N)\times \on{Gr}(\omega_{\mathsf{S}}).\]
In particular, we recover the Weinstein splitting theorem. Section \ref{subsec:33} is devoted to a proof of Theorem \ref{th:fre}, after 
some preliminary results about normal bundles. These two sections are based on \cite{bur:spl}, to which we refer for more detailed discussions.

\subsection{Normal bundles}\label{subsec:nor}
Consider the category whose objects are the pairs $(M,N)$ of a manifold $M$ with a submanifold $N$, and whose morphisms 
$(M_1,N_1)\to (M_2,N_2)$
are the smooth maps $\Phi\colon M_1\to M_2$ taking $N_1$ to $N_2$. The normal bundle functor $\nu$ assigns to $(M,N)$ the 
vector bundle 
\[ \nu(M,N)=TM|_N/TN,\]
over $N$, and to a morphism $\Phi\colon (M_1,N_1)\to (M_2,N_2)$ the vector bundle morphism 
\[ \nu(\Phi)\colon \nu(M_1,N_1)\to \nu(M_2,N_2)\]
induced by the tangent map $T\Phi\colon TM_1\to TM_2$. The normal functor is compatible with the tangent functor; in particular $\nu(TM,TN)\cong T\nu(M,N)$ as double vector bundles, and $\nu(T\Phi)=T\nu(\Phi)$. 

For a vector bundle $E\to M$, there is a canonical identification $\nu(E,M)\cong E$. 
In particular, given a pair $(M,N)$ we have the identification $\nu(\nu(M,N),N)=\nu(M,N)$. 
A \emph{tubular neighborhood embedding} is a map of pairs 
\[ \varphi\colon (\nu(M,N),N)\to (M,N)\]
such that $\varphi\colon \nu(M,N)\to M$ is an embedding as an open subset, and 
the map $\nu(\varphi)$ is the identity. 

Suppose $X$ is a vector field on $M$ such that $X|_N$ is tangent to $N$. 
Viewed as a section of the tangent bundle, it defines a morphism $X\colon (M,N)\to (TM,TN)$, inducing a section $\nu(X)\colon \nu(M,N)\to T\big(\nu(M,N)\big)$. This vector field 
\[ \nu(X)\in \mf{X}\big(\nu(M,N)\big)\]
is called the \emph{linear approximation} to $X$ along $N$. In local bundle trivializations, the 
linear approximation is the first order Taylor approximation in the normal directions. 
By a \emph{linearization} of the vector field $X$, we mean a tubular neighborhood embedding $\varphi$ 
taking $\nu(X)$ to $X$ on a possibly smaller neighborhood of $N$.  The problem of $C^\infty$-linearizability of vector fields is quite subtle; the main result (for $N=\pt$) is the \emph{Sternberg linearization theorem} \cite{ste:loc} which proves existence of linearizations under non-resonance conditions. 
\begin{definition}\cite{bur:spl}
A vector field $X\in\mf{X}(M)$ is called \emph{Euler-like along $N$} if it is complete, with $X|_N=0$, and with linear approximation $\nu(X)$ the Euler vector field $\E$ on $\nu(M,N)$. 
\end{definition}
An Euler-like vector field determines a tubular neighborhood embedding: 
\begin{lemma} 
If $X\in\mf{X}(M)$ is Euler-like along $N$, then $X$ determines a unique tubular neighborhood embedding $\varphi\colon \nu(M,N)\to M$ such that 
\[ \E\sim_{\varphi} X.\]
\end{lemma}
\begin{proof} 
The main point is to show that $X$ is linearizable along $N$. Pick any 
tubular neighborhood embedding to assume $M=\nu(M,N)$. Since $\nu(X)=\E$, 
it follows that  the difference $Z=X-\E$ vanishes to second order along $N$. Let $\kappa_t$ denote scalar multiplication by $t$ on $\nu(M,N)$, and consider the family of vector fields, defined for $t\neq 0$,  
\[ Z_t=\f{1}{t} \kappa_t^* Z\]
Since $Z$ vanishes to second order along $N$, this is well-defined even at $t=0$. Let 
$\phi_t$ be its (local) flow; thus $\f{d}{d t}(\phi_t)_*=(\phi_t)_*\circ \L_{Z_t}$ as operators on tensor fields (e.g., functions, vector fields, and so on). 
On a sufficiently small open neighborhood of $N$ in $\nu(M,N)$, it is defined for all $|t|\le 1$. The flow of the Euler vector field $\E$ is  $s\mapsto \kappa_{\exp(-s)}$; by substitution $t=\exp(-s)$ this shows that 
$\f{d}{d t}\kappa_t^*=t^{-1}\kappa_t^*\circ \L_\E$.  Consequently, 
\[ \f{d}{d t}(t Z_t)=
\f{d}{d t}(\kappa_t^* Z)=
\ca{L}_\E Z_t=[\E,Z_t].\]
Therefore, 
\[ \f{d}{d t}(\phi_t)_*(\E+t Z_t)=(\phi_t)_*(\L_{Z_t}(\E+t Z_t)-[\E,Z_t])=0,\]
which shows that $(\phi_t)_*(\E+t Z_t)$ does not depend on $t$. Comparing the values at $t=0$ and $t=1$ we obtain $(\phi_1)_* X=\E$, so that $(\phi_1)^{-1}$ giving the desired linearization on a neighborhood of $N$. In summary, this shows that there exists a map from a neighborhood of the zero section of $\nu(M,N)$ to a neighborhood of $N$ in $M$,  intertwining the two vector fields $\E$ and $X$, and hence also their flows. Since $X$ is complete, we may use the flows to extend globally to a tubular neighborhood embedding of the full normal bundle. 
\end{proof}

\subsection{Cosymplectic submanifolds}\label{subsec:33}
 Let $(M,\pi)$ be a Poisson manifold, and $i\colon N\hra  M$ a cosymplectic submanifold. If $\alpha\in \Omega^1(M)$ vanishes along $N$, then $X=\pi^\sharp(\alpha)$ 
vanishes along $N$, and so we can consider its linear approximation $\nu(X)$. 
It is an \emph{infinitesimal gauge transformation} of $\nu(M,N)$: a fiberwise linear vector field 
whose restriction to $N$ is zero. (Equivalently, it is tangent to the fibers.) By transversality of $\pi^\sharp\colon T^*M\to TM$ to $N\subset M$, \emph{any}
such vector field on $\nu(M,N)$ can be realized as the linear approximation of $\pi^\sharp(\alpha)$ for some $\alpha$ with $\alpha|_N=0$. (For details, see \cite{bur:spl}, proof of Lemma 3.9.)
In particular, this applies to the Euler vector field $\E$ on $\nu(M,N)$. By multiplying the 1-form $\alpha$ by a bump function supported near $N$, we can arrange that 
$X=\pi^\sharp(\alpha)$ is complete. This shows: 
\begin{lemma}
There exists a 1-form $\alpha\in \Om^1(M)$, with $\alpha|_N=0$, such that the vector field $X=\pi^\sharp(\alpha)$ is Euler-like. 
\end{lemma}
We are now in position to prove Frejlich-M\u{a}rcu\c{t}'s normal form theorem. 

\begin{proof}[Proof of the Frejlich-M\u{a}rcu\c{t} theorem \ref{th:fre}, after \cite{bur:spl}]
Choose a 1-form $\alpha$ as in the Lemma. 
The Euler-like vector field $X=\pi^\sharp(\alpha)$ gives a tubular neighborhood embedding 
$\psi\colon \nu(M,N)\to M$, taking $\E$ to $X$. 
Using this embedding, we may assume $M=\nu(M,N)$ is a vector bundle, with $X=\E$ the Euler vector field. The flow of $X$ is thus  given as $\Phi_t=\kappa_{\exp(-t)}$. 
Consider the infinitesimal automorphism 
$(\d\alpha,X)$ of $\T M$ defined by $\sigma =X+\alpha\in \Gamma(\on{Gr}(\pi))$. 
By Proposition \ref{prop:integrate}, the corresponding 1-parameter group of automorphisms is $(-\omega_t,\Phi_t)$, where 
\[ \omega_t=-\d \int_0^t (\Phi_s)_*\alpha\ \d s=
-\d \int_0^t (\Phi_{-s})^*\alpha\ \d s
=\d\int_{\exp(t)}^1\ \f{1}{v} \kappa_v^* \alpha\ \d v
.\]
Since $\sigma$ is a section of $\on{Gr}(\pi)$, the action $\ca{R}_{\omega_t}\circ \T \Phi_t$ of this 1-parameter group preserves $\on{Gr}(\pi)$. That is, 
\[ \ca{R}_{\omega_t} \big((\Phi_{-t})^! \on{Gr}(\pi)\big)=\on{Gr}(\pi)\]
for all $t\ge 0$. Consider the limit $t\to -\infty$ in this equality.  Since $\alpha$ vanishes along $N$, the family of forms $\f{1}{v} \kappa_v^* \alpha$  extends smoothly to $v=0$. Hence $\omega:=\omega_{-\infty}$ is well-defined:  
\[ \omega=\d \int_0^1 \f{1}{v}\,\, \kappa_v^* \alpha\ \d v.\]
On the other hand, 
\[ \Phi_{\infty}=\kappa_0=i\circ p\] 
where $p\colon \nu(M,N)\to N$ is the projection, and $i\colon N\to M$ is the inclusion. Thus 
\[  \ca{R}_\omega(p^! i^! \on{Gr}(\pi))=\on{Gr}(\pi).\]
Since $i^!\on{Gr}(\pi)=\on{Gr}(\pi_N)$, the left hand side is $\ca{R}_\omega(p^!\on{Gr}(\pi_N))$, which coincides with the model Poisson structure $\pi_{\mathsf{V}}$ near $N$.
\end{proof}
\vskip.06in

\subsection{Notes} 
The Weinstein splitting theorem was proved in Weinstein's paper \cite{wei:loc}, using an inductive argument to put a given Poisson structure into a partial normal form. There is also a $G$-equivariant version, for an action of a compact Lie group preserving the Poisson structure. However, this version does not readily follow from the original proof, since the inductive argument does not lend itself to `averaging'. In \cite{mir:no}, the $G$-equivariant case was proved under some assumptions, and in \cite{fre:nor} in 
generality. The argument presented here, using `Euler-like' vector fields, is taken from 
\cite{bur:spl}, where a similar idea was used to prove or re-prove splitting theorems in various other settings.

\section{The Karasev-Weinstein symplectic realization theorem}\label{sec:karasevweinstein}

\subsection{Symplectic realization}
Our starting point is the following definition, due to Weinstein. 
\begin{definition}\cite{wei:loc} A \emph{symplectic realization} of a Poisson manifold 
$(M,\pi_M)$ is a symplectic manifold $(P,\omega_P)$ together with a Poisson map 
$\varphi\colon P\to M$. It is called a \emph{full} symplectic realization if $\varphi$ is a surjective submersion.  
\end{definition}
An example of a symplectic realization is the inclusion of a symplectic leaf. However, we will mainly be interested in full symplectic realizations. 
\begin{examples}\label{ex:I}
\begin{enumerate}
\item 
(See \cite{ca:ge}.) Let $M=\R^2$ with the Poisson structure $\pi=x\f{\p}{\p x}\wedge \f{\p}{\p y}$. A full symplectic realization is given by $P=T^*\R^2$ with the standard symplectic form $\omega=\sum_{i=1}^2\d q_i\wedge \d p_i$, and 
\[ \varphi(q_1,q_2,p_1,p_2)=(q_1,q_2+p_1 q_1).\]
\item
Let $G$ be a Lie group, with Lie algebra $\g$. The space $M=\g^*$, with the Lie-Poisson structure, has a full symplectic realization $P=T^*G$, with the standard symplectic structure, and with the map $\varphi$ given by left trivialization. (Example (a) may be seen as a special case, using that $x\f{\p}{\p x}\wedge \f{\p}{\p y}$ is a linear Poisson structure, corresponding to a 2-dimensional Lie algebra.) 
\item
Let $M$ be a manifold with the zero Poisson structure. Then the cotangent bundle, with its standard symplectic structure, 
and with $\varphi$ the cotangent projection $\tau\colon T^*M\to M$, is a full symplectic realization. 
\item 
Let $M$ be a manifold with a symplectic structure. Then $P=M$, with $\varphi$ the identity map, is a full symplectic realization. 
\item 
Let $P$ be a symplectic manifold with a proper, free action of a Lie group $G$, and 
$M=P/G$  the quotient manifold with the induced Poisson structure. Then $P$ is a 
symplectic realization of $M$, with $\varphi$ the quotient map. 
\end{enumerate}
\end{examples}
Does every Poisson manifold admit a full symplectic realization? Before addressing this question, let us first consider the opposite 
problem: when does a symplectic structure descend under a surjective submersion?

\begin{proposition}
Let $(P,\pi_P)$ be a Poisson manifold, and $\varphi\colon P\to M$ a surjective submersion with 
connected fibers. 
\begin{enumerate}
\item
Then $\pi_P$ descend to a Poisson structure $\pi$ on $M$ if and only if the conormal bundle to the fibers of $\varphi$, 
\[ \on{ann}(\ker T \varphi)\subset T^*P,\]
is a subalgebroid of the cotangent Lie algebroid. 
\item (Libermann's theorem \cite{lib:pro}.)
If $\pi_P$ is the Poisson structure for a symplectic form (i.e., $\pi_P^\sharp=-(\omega^\flat)^{-1}$), then 
$\pi_P$ descends if and only if the $\omega_P$-orthogonal 
distribution to $\ker(T \varphi)$ is involutive. 
\end{enumerate}
\end{proposition}
\begin{proof}
(a) The bundle $\on{ann}(\ker T \varphi)$ is spanned by all $\d\  \varphi^* f$ with $f\in C^\infty(M)$. 
If  $\pi_P$ descends to $\pi_M$, then $\varphi$ is a Poisson map, with respect to $\pi_P,\ \pi_M$. 
Hence 
\[ [\d\ \varphi^* f,\ \d\ \varphi^* g]=\d\ \{ \varphi^* f,\ \varphi^* g\}_P =\d\ \varphi^*\{ f,\ g\}_M,\]
proving that $ \on{ann}(\ker T \varphi)$ is a sub-Lie algebroid. Conversely, if 
$\on{ann}(\ker T \varphi)$ is a sub-Lie algebroid, it follows that for all $f,g\in C^\infty(M)$, the differential 
$\d \{ \varphi^* f,\ \varphi^* g\}_P$ vanishes on $\ker(T \varphi)$. Since the fibers of $\varphi$ are connected, this means that the function 
$\{ \varphi^* f,\ \varphi^* g\}_P$ is fiberwise constant, and hence is the pull-back of a function 
on $M$. Take this function to be the definition of $\{f,g\}_M$. It is straightforward to see that 
$\{\cdot,\cdot\}_M$ is a bi-derivation, hence it corresponds to a bivector field $\pi_M$ such that $\pi_P\sim_\varphi \pi_M$. 
(b)  The map $\pi_P^\sharp \colon T^*P\to TP$ is a Lie algebroid isomorphism, taking
$\on{ann}(\ker(T \varphi))$ to  the $\omega_P$-orthogonal bundle 
of $\ker(T \varphi)$. The latter being a Lie subalgebroid is equivalent to Frobenius integrability.
\end{proof}

Libermann's theorem shows that if $\varphi\colon P\to M$ is a full symplectic realization, then the foliation given by the $\varphi$-fibers is symplectically orthogonal to another foliation. If the leaf space for this second foliation is smooth, then the latter inherits a Poisson structure, again by application of Liberman's theorem. This suggests the following definition: 
\begin{definition}\cite{wei:loc} \label{def:dualpair}
A \emph{dual pair} of Poisson manifolds $(M_1,\pi_1)$ and $(M_2,\pi_2)$ is a symplectic manifold $(P,\omega)$ together with two surjective submersions $\tz\colon P\to M_1,\ \sz\colon P\to M_2$ such that 
\begin{itemize}
\item $\tz$ is a Poisson map, 
\item $\sz$ is an anti-Poisson map,
\item $\ker(T\tz)$ is the $\omega$-orthogonal bundle to 
$\ker(T\sz)$. 
\end{itemize}
\end{definition}
Note that, in particular, $\dim P=\dim M_1+\dim M_2$. One thinks of the \emph{correspondence manifold} $P$ as a generalized morphism  from $(M_2,\pi_2) \to (M_1,\pi_1)$.  (For more on this viewpoint, see \cite{xu:mor}.)

\begin{examples} Let us illustrate the concept of dual pairs for the Examples \ref{ex:I}. We will write 
$M_1\leftarrow P \rightarrow M_2$ to indicate a dual pair. In most cases, we take $M_1=M_2=M$ (which is the setting for the Karasev-Weinstein theorem \ref{th:kw}). 
\begin{enumerate}
\item 
For $M=\R^2$ with $\pi=x\f{\p}{\p x}\wedge \f{\p}{\p y}$, we obtain a dual pair 
$M\leftarrow P \rightarrow M$ by letting  $P=T^*(R^2)$ with the standard symplectic form, 
\[ \tz(q_1,q_2,p_1,p_2)=(q_1,q_2+p_1 q_1),\ \ \ \ \sz(q_1,q_2,p_1,p_2)=(q_1\exp(p_1),q_2).\]
%
\item
For $M=\g^*$ with the standard Poisson structure, we have a dual pair 
$\g^* \leftarrow T^*G \rightarrow \g^*$, with $\tz$ the left trivialization, and $\sz$ the right trivialization. 
\item 
For $M$ with the zero Poisson structure, we have a dual pair $M\leftarrow T^*M \rightarrow M$ with $\tz=\sz=\tau$ the cotangent projection. 
\item 
For a symplectic manifold $M$, we obtain a  dual pair $M\leftarrow M\times M^-\rightarrow M$, with $\tz,\sz$ the projections from $M\times M^-$ to the two factors. 
Another dual pair is $M\leftarrow M \rightarrow \pt$. 
\item 
Let $P$ be a symplectic manifold with a free, proper action of a Lie group $G$, admitting a moment map $\Phi\colon P\to \g^*$ in the sense of symplectic geometry. Then we obtain a dual pair $P/G \leftarrow P\rightarrow \Phi(P)^-$ , where $\tz$ is the quotient map and $\sz=\Phi$. Here we are using the well-know fact from symplectic geometry that a moment map for a free action is a submersion. 
\end{enumerate}
\end{examples}

We finally state the main result concerning the existence of symplectic realizations. 
\begin{theorem}[Karasev \cite{kar:ana}, Weinstein \cite{wei:loc}] \label{th:kw}
Let $(M,\pi)$ be a Poisson manifold. Then there exists a 
symplectic manifold $(P,\omega)$, with an inclusion $i\colon M\hra P$ as a Lagrangian submanifold, and with two surjective submersions 
$\tz,\sz\colon P\to M$ such that $\tz\circ i=\sz\circ i=\on{id}_M$, and
such that 
\begin{itemize}
\item $\tz$ is a Poisson map, 
\item $\sz$ is an anti-Poisson map,
\item The $\tz$-fibers and $\sz$-fibers are $\omega$-orthogonal. 
\end{itemize}
\end{theorem}
In fact, much more is true: There exists a structure of a local symplectic groupoid on $P$, having $\sz,\tz$ as the source and target maps, and $i$ as the inclusion of units. See Section \ref{subsec:sympgr} below.

\subsection{The Crainic-M\u{a}rcu\c{t} formula}\label{subsec:cramar}

Our proof of Theorem \ref{th:kw} will use an explicit construction of the realization due to Crainic-M\u{a}rcu\c{t} \cite{cra:exi}, with later simplifications due to Frejlich-M\u{a}rcu\c{t} \cite{fre:rea}.  As the total space $P$ for the full symplectic realization, we will take a suitable open neighborhood of 
$M$ inside the cotangent bundle  
\[ \tau \colon T^*M\to M.\]
\begin{definition}[Crainic-M\u{a}rcu\c{t} \cite{cra:exi}] Let $(M,\pi)$ be a Poisson manifold. 
A vector field $X\in \mf{X}(T^*M)$ is called a \emph{Poisson spray} if it is homogeneous of degree $1$ in fiber directions, and for all $\mu\in T^*M$, 
\[ (T_\mu \tau)(X_\mu)=\pi^\sharp(\mu).\]
\end{definition}
The homogeneity requirement means that $\kappa_t^* X= t X$, where $\kappa_t$ is fiberwise multiplication by $t\neq 0$. 
In local coordinates, given a Poisson structure 
\begin{equation}\label{eq:piloc}
\pi=\hh \sum_{ij}\pi^{ij}(q) \f{\p}{\p q^i}\wedge \f{\p}{\p  q^j},
\end{equation}
the Poisson sprays are vector fields of the form 
\begin{equation}\label{eq:spray}
 X=\sum_{ij} \pi^{ij}(q)\, p_i\, \f{\p}{\p  q^j}+\hh \sum_{ijk}
\Gamma^{ij}_k(q)\, p_i\, p_j  \, \f{\p}{\p  p^k}\end{equation}
where $p_i$ are the cotangent coordinates, and  
$\Gamma^{ij}_k=\Gamma^{ji}_k$ are functions. 
\begin{lemma}
Every Poisson manifold $(M,\pi)$ admits a Poisson spray. 
\end{lemma}
\begin{proof}
In local coordinates, Poisson sprays can be defined by the formula above (e.g., with $\Gamma^{ij}_k=0$). To obtain a global Poisson spray, one patches these local definitions together, using a partition of unity on $M$.
\end{proof}
Let $X$ be a Poisson spray, and $\Phi_t$ its local flow. Since $X$ vanishes along $M\subset T^*M$, there exists an open neighborhood of 
$M$ on which the flow is defined for all $|t|\le 1$. On such a neighborhood, put
\[ \omega=\int_0^1 (\Phi_s)_*\omega_{\on{can}}\ \d s,\]
where $\omega_{\on{can}}$ is the standard symplectic form of the cotangent bundle. 
\begin{lemma}
The 2-form $\omega$ is symplectic along $M$. 
\end{lemma}
\begin{proof}
For $m\in M\subset T^*M$, consider the decomposition
\[ T_m(T^*M)=T_mM\oplus T^*_mM.\]
Since the vector field $X$ is homogeneous of degree $1$, it vanishes along $M$. In particular, its flow 
$\Phi_t$ fixes $M\subset P$, hence $T\Phi_t$ is a linear transformation of $T_m(T^*M)$.  Consequently, 
$(T_m\Phi_t)(v)=v$ for all $m\in M$ and $v\in T_mM$. Again by homogeneity, the linear approximation along $M\subset T^*M$ vanishes: $\nu(X)=0$ as a vector field on $\nu(T^*M,M)\cong T^*M$. 
(This is not to be confused with the linear approximation of $X$ at $\{m\}$, which may be non-zero.) 
Consequently,  $\nu(\Phi_t)=\on{id}_{T^*M}$, which shows that 
\[ (T_m\Phi_t)(w)=w\mod T_mM\]
for all $w\in T_m^*M$. Hence $((\Phi_s)_*\omega_{\on{can}})(v,\cdot)=\omega_{\on{can}}(v,\cdot)$ for all $v\in T_mM$, and therefore
\[ \omega(v,\cdot)=\om_{\on{can}}(v,\cdot).\]
Since $TM$ is a Lagrangian subbundle (in the symplectic sense!) of $T(T^*M)|_M$ with respect to $\omega_{\on{can}}$, this implies that the 2-form $\omega$ is 
symplectic along $M$.
\end{proof}

\begin{theorem}[Crainic-M\u{a}rcu\c{t} \cite{cra:exi}, Frejlich-M\u{a}rcu\c{t} \cite{fre:rea}] \label{th:cm}
Let $P\subset T^*M$
be an open neighborhood of the zero section, with the property that 
$\Phi_t(m)$ is defined for all $m\in P$ and $|t|\le 1$, and such that $\omega$ is symplectic on $P$. Let $i\colon M\hra P$ be the inclusion as the zero section, and put 
\[ \sz=\tau,\ \ \ \ \ \ \tz=\tau\circ \Phi_{-1}.\] 
Then the symplectic manifold $(P,\omega)$ together with the maps $\tz,\sz, i$ has the properties from the Karasev-Weinstein theorem \ref{th:kw}.
\end{theorem}

\subsection{Proof of Theorem \ref{th:cm}}
Following \cite{fre:rea}, we will give a proof of this result using Dirac geometry. We begin with the following Dirac-geometric reformulation of the dual pair condition, Definition 
\ref{def:dualpair}. 

\begin{proposition}[Frejlich-M\u{a}rcu\c{t} \cite{fre:rea}] 
\label{prop:fm}
Let $(M_1,\pi_1)$ and $(M_2,\pi_2)$ be two Poisson manifolds, and $(P,\omega)$ 
a symplectic manifold with surjective submersions $\tz\colon P\to M_1,\ 
\sz\colon P\to M_2$. Then 
\[ M_1\leftarrow P \rightarrow M_2\] 
is a dual pair if and only if $\dim P=\dim M_1+\dim M_2$ and 
\begin{equation}\label{eq:thecondition}
 \ca{R}_\omega(\tz^! \on{Gr}(\pi_1))=\sz^!\on{Gr}(\pi_2).\end{equation}
\end{proposition}
\begin{proof}
Suppose first that the condition holds. Let $v_1\in \ker(T\tz),\ v_2\in\ker(T\sz)$.  Using 
$\ker(T\tz)\subset \tz^! \on{Gr}(\pi_1)$, and similarly for $\sz$, 
we obtain 
\[ \omega(v_1,v_2)=-\l \ca{R}_\omega(v_1),\ v_2\r=0,\]
since both $\ca{R}_\omega(v_1)$ and $v_2$ are in $\sz^!\on{Gr}(\pi_2)$. By dimension count, this shows that $\ker(T\tz)$ is the $\omega$-orthogonal bundle to 
$\ker(T\sz)$. To show that $\tz$ is a Poisson map, let $\mu_1\in T^*_{\tz(p)}M_1$ be given. Define $v\in T_pP$ by $\iota_v\omega=\tz^*\mu_1$. Since $\tz^*\mu_1$ annihilates all elements of $\ker(T\tz)$, we see 
that $v$ lies in the $\omega$-orthogonal bundle:  $v\in \ker(T\sz)$. But then 
\[ v+\iota_v\omega=\ca{R}_{-\omega}(v)\in \ca{R}_{-\omega}(\sz^! \on{Gr}(\pi_2))=\tz^!\on{Gr}(\pi_1).\]
That is, $v+\tz^*\mu_1=
v+\iota_v\omega$ is $\tz$-related to an element of $\on{Gr}(\pi_1)$. Necessarily, this element is $\pi_1^\sharp(\mu_1)+\mu_1$. This shows $T_p\tz (v)=
\pi_1^\sharp(\mu_1)$, proving that $\tz$ is Poisson. Similarly, $\sz$ is anti-Poisson. 

Conversely, suppose that $(\tz,\sz)\colon P\to M_1\times M_2^-$ is a Poisson map. We have a direct sum decomposition
\begin{equation}\label{eq:decompose}
\tz^!\on{Gr}(\pi_1)=\ker(T\tz)\oplus ( \tz^!\on{Gr}(\pi_1)\cap \on{Gr}(\pi_P));\end{equation}
where the elements of the second summand are those of the form 
$\pi_P^\sharp(\tz^*\mu)+\tz^*\mu$, and similarly for 
$\sz^!\on{Gr}(\pi_2)$.  Let 
\[ v+\tz^*\mu\in  \tz^!\on{Gr}(\pi_1)\cap \on{Gr}(\pi_P).\] 
Then $\tz^*\mu=\pi_P^\sharp \mu=-\iota_v\omega$, hence 
$\ca{R}_\omega(v+\tz^*\mu)=v$. 
On the other hand, for all $w\in \ker(T\tz)$ we have that 
\[ \omega(v,w)=-\l v+\tz^*\mu,\ w\r=0\]
(since  $\tz^!\on{Gr}(\pi_1)$ is Lagrangian), which shows that $v\in \ker(T\sz)$. We conclude that $\ca{R}_\omega$ restricts to an isomorphism
\[ \ca{R}_\omega\colon \tz^!\on{Gr}(\pi_1)\cap \on{Gr}(\pi_P)\to \ker(T\sz);\]
similarly $\ca{R}_{-\omega}$ restricts to an isomorphism 
from $\sz^!\on{Gr}(\pi_2)\cap \on{Gr}(\pi_P)$ to $\ker(T\tz)$. In summary, 
$\ca{R}_\omega$ restricts to an isomorphism from $\tz^!\on{Gr}(\pi_1)$ to 
$\sz^!\on{Gr}(\pi_2)$. 
\end{proof}

With this result in place, we can prove the main result: 

\begin{proof}[Proof of Theorem \ref{th:cm}]  Let 
$\alpha\in \Om^1(T^*M)$ be the canonical 1-form of the cotangent bundle. That is,  for all $\mu\in T^*M$, 
$ \alpha_\mu=(T_\mu\tau)^*\mu$. 
Recall that $\omega_{\on{can}}=-\d\alpha$. \footnote{In local cotangent coordinates, $\alpha=\sum_i p_i\d q^i$ and $\omega_{\on{can}}=\sum_i \d q^i\wedge \d p_i$.} Given a Poisson spray $X$, observe that the section $X+\alpha$ of $\T (T^*M)$ takes values in $\tau^!\on{Gr}(\pi)$:
\[ X+\alpha\in \Gamma(\tau^!\on{Gr}(\pi))\subset \Gamma(\T (T^*M)).\] Indeed, the definition of a spray (and of the canonical 1-form $\alpha$) means precisely that for all $\mu\in T^*M$, 
\[ X_\mu+\alpha_\mu\ \sim_\tau\  \pi^\sharp(\mu)+\mu,\]
as required. Since $\tau^!\on{Gr}(\pi)$ is a Dirac structure, the infinitesimal automorphism $(\d\alpha,X)\in\mf{aut}\big(\T(T^*M)\big)$ defined by the section $X+\alpha$ 
preserves $\tau^!\on{Gr}(\pi)$. By Proposition \ref{prop:couraut}, the (local) 1-parameter group of automorphisms exponentiating
$(\d\alpha,X)\in\mf{aut}\big(\T(T^*M)\big)$ is given by $(-\omega_t,\Phi_t)$, where 
$\Phi_t$ is the (local) flow of $X$, and 
\[ \omega_t=-\d \int_0^t (\Phi_s)_*\alpha=\int_0^t (\Phi_s)_*\omega_{\on{can}}.\]
We conclude 
\[ \ca{R}_{\omega_t}\circ \T \Phi_t (\tau^! \on{Gr}(\pi))=\tau^!\on{Gr}(\pi).\]
Putting $t=1$ in this identity, and using the definition of $\tz,\sz,\omega$, together with the fact that 
$T\Phi_1(E)=(\Phi_{-1})^!E$ for any Dirac structure $E\subset \T(T^*M)$, we obtain 
$\ca{R}_\omega\circ \tz^!\on{Gr}(\pi)=\sz^!\on{Gr}(\pi)$. By Proposition \ref{prop:fm}, we are done. 
\end{proof}

\subsection{The local symplectic groupoid}\label{subsec:sympgr} 
Recall that a  \emph{Lie groupoid} $\G\rra M$ is given by a manifold $\G$ of \emph{arrows}, a submanifold $\G^{(0)}=M$ of \emph{units}, two surjective 
submersions  $\sz,\tz\colon \G\to M$ called the \emph{source} and \emph{target} maps, 
satisfying $\sz\circ i=\tz\circ i=\on{id}_M$, 
and a partially defined multiplication map 
\[ \Mult_\G\colon \G^{(2)}=\G\ _{\sz}\times_{\tz} \G \to \G,\ \ \ (g_1,g_2)\mapsto g_1\circ g_2\]
satisfying certain axioms of associativity, neutral element, and existence of inverses.  See Appendix \ref{app:groupoids} for details. The entire groupoid structure is encoded in the 
graph of the groupoid multiplication 
\[ \on{Gr}(\Mult_\G)=\{(g_1\circ g_2,g_1,g_2)\in\G^3|\ (g_1,g_2)\in \G^{(2)}\}\]
\begin{definition}[Weinstein \cite{wei:sym}]
A \emph{symplectic groupoid} is a Lie groupoid 
$\G\rra M$,  equipped with a symplectic structure $\om_\G$ such that that the graph of the groupoid multiplication $ \on{Gr}(\Mult_\G)$ is a Lagrangian submanifold 
of $\G\times \ol{\G\times \G}$. 
\end{definition}
Here the line indicates the opposite symplectic structure; thus $\G\times \ol{\G\times \G}$ has the symplectic structure $ \pr_1^*\omega-\pr_2^*\omega-\pr_3^*\omega$, 
with $\pr_i$ the projection to the $i$-th factor. Once again, `Lagrangian' is meant in the sense of symplectic geometry: $\on{Gr}(\Mult_\G)$ has the middle dimension $\f{3}{2}\dim \G$, and the pullback of the symplectic form vanishes. We leave it as an exercise to show that the 
space of units of a symplectic groupoid is Lagrangian. 

As proved in Coste-Dazord-Weinstein \cite{cos:gro}, for any symplectic groupoid 
$(\G,\omega)$ the space $M$ of objects acquires a Poisson structure, in such a way 
that $\tz,\sz,i$ satisfy the Karasev-Weinstein conditions from Theorem \ref{th:kw}. That is, the map 
\[ (\tz,\sz)\colon \G\to M\times M^-\] 
is Poisson, where $M^-$ stands for $M$ with the opposite Poisson structure $-\pi$.
One calls $(\G,\omega)$ a \emph{symplectic groupoid integrating} $(M,\pi)$. 

\begin{example}
If $(M,\omega)$ is a symplectic manifold (regarded as a Poisson manifold), then the pair groupoid $\ol{M}\times M\rra M$ (cf. Appendix \ref{app:groupoids})
is a symplectic groupoid integrating $M$. Here $\ol{M}$ indicates $M$ with the opposite symplectic structure $-\omega$.
\end{example}

\begin{example}\cite{cos:gro} \label{ex:cotangentgroupoid}
Suppose $G$ is a Lie group, with Lie algebra $\g$. Then the cotangent  bundle $T^*G$, with its standard symplectic structure, has the structure of a symplectic groupoid 
\[ T^*G\rra \g^*\]
where for $\mu_i\in T^*_{g_i}G,\ \ i=1,2$ and $\mu\in T^*_gG$, 
\[ \mu=\mu_1\circ \mu_2\ \ \ \Leftrightarrow\ \ \ g=g_1g_2,\ \ (\mu_1,\mu_2)=(T_{g_1,g_2}\Mult_G)^*\mu.\] 
Here $\Mult_G\colon G\times G\to G$ is the group multiplication map of the group $G$. The space of 
units is $\g^*$ embedded as the fiber $T^*_eG$, while  source and target map are given by left trivialization, respectively right trivialization. The symplectic groupoid $T^*G$ integrates $\g^*$ equipped with the 
Lie-Poisson structure. Indeed, the map $\tz$ is the symplectic moment map for the 
cotangent lift of the left $G$-action on itself, and moment maps are always Poisson maps. 
\end{example}

Not every Poisson manifold admits an integration to a symplectic groupoid. But one always has a \emph{local} symplectic groupoid integrating the Poisson structure. Local Lie groupoids are a generalization of Lie groupoids, where the groupoid multiplication is only defined 
on some open neighborhood of the diagonally embedded $M\subset \G^{(2)}$. 
(For details, see \cite[Definition III.1.2]{cos:gro}.)  

\begin{theorem}[Coste-Dazord-Weinstein \cite{cos:gro}, Karasev \cite{kar:ana}] Every Poisson manifold $(M,\pi)$ 
admits a local symplectic groupoid $\G\rra M$ integrating it. 
\end{theorem}
 In fact,  it is shown in \cite{cos:gro} that any full symplectic realization $\tau\colon P\to M$ for which $\tau$ is a retraction onto a Lagrangian submanifold $i\colon M\hra P$, can be given the structure of a local symplectic groupoid over $M$. The following discussion is similar to their approach. It shows that given the data from the Karasev-Weinstein 
 theorem, one automatically gets the local groupoid structure. 
\begin{proposition}
Let $(M,\pi_M)$ be a Poisson manifold, and $(P,\omega_P)$ a symplectic manifold with an inclusion 
$i\colon M\to P$ as a Lagrangian submanifold, and with  two surjective submersions 
\[ \tz, \sz\colon P\to M,\]
satisfying the properties of the Karasev-Weinstein 
 theorem \ref{th:kw}. Replacing $P$ with a smaller neighborhood of $M$ if necessary, it  has a unique structure of a local symplectic groupoid, with $\sz$ and $\tz$ as the source and target map, and $i$ the inclusion of units. 
\end{proposition}

The idea of proof is to generate the graph of the groupoid multiplication by flows of vector fields, playing the 
role of left-invariant and right-invariant vector fields. For all $\alpha\in  \Omega^1(M)$, define vector fields $\alpha^L,\alpha^R\in \mf{X}(P)$ by 
\[ \alpha^R=-\pi_P^\sharp(\tz^*\alpha),\ \ \alpha^L=-\pi_P^\sharp(\sz^*\alpha).\]
(See \cite[Section II.2]{cos:gro}.) 

\begin{proposition}[Left- and right-invariant vector fields]\label{prop:lr} The vector fields 
$\alpha^L$ (resp. $\alpha^R$) span the tangent spaces to the $\tz$-fibers 
(resp. $\sz$-fibers) everywhere. One has  
\[ 
\alpha^L\sim_{\sz} \pi_M^\sharp(\alpha),\ \ \ 
\alpha^R\sim_{\tz} -\pi_M^\sharp(\alpha)\ \ \ \ 
\]
and the Lie bracket relations 
\[ [\alpha^L,\beta^L]=[\alpha,\beta]^L,\ \ \ 
[\alpha^R,\beta^R]=-[\alpha,\beta]^R,\ \ \  [\alpha^L,\beta^R]=0\]
for all $\alpha,\beta\in \Omega^1(M)$. Furthermore,
\[ \omega_P(\alpha^L,\beta^L)=-\sz^* \pi_M(\alpha,\beta),\ \ \ \ 
\omega_P(\alpha^R,\beta^R)=\tz^* \pi_M(\alpha,\beta),\ \ \ \ \ 
\omega_P(\alpha^L,\beta^R)=0.\]
\end{proposition}
\begin{proof}
Observe that the three conditions from Theorem \ref{th:kw} can be combined into the 
single condition that $(\tz,\sz)\colon P\to M\times M^-$ is a Poisson map. 
Write the definition of $\alpha^L,\alpha^R$ as 
\begin{equation}\label{eq:canthink}
 \alpha^L=-\pi_P^\sharp \big((\tz,\sz)^*(0,\alpha)\big),\ \ \ \ 
\alpha^R=\pi_P^\sharp\big((\tz,\sz)^*(\alpha,0)\big).
\end{equation}
Since $(\tz,\sz)$ is Poisson with respect to $(\pi_M,-\pi_M)$, it follows that 
$\alpha^L\sim_{\tz}0$ and $\alpha^L\sim_{\sz}\pi_M^\sharp(\alpha)$, and similarly for $\alpha^R$. It also implies that pullback under $(\tz,\sz)$ preserves Lie brackets on 1-forms, thus 
\[ 
[\tz^*\alpha,\tz^*\beta]=\tz^*[\alpha,\beta],\ \ \ \ 
[\sz^*\alpha,\sz^*\beta]=-\sz^*[\alpha,\beta],\ \ \ \ 
[\sz^*\alpha,\tz^*\beta]=0.\]
Applying $\pi_P^\sharp$ to both sides, one obtains the Lie bracket relations between the vector fields $\alpha^L,\ \alpha^R$. Next, 
\begin{align*}
\omega_P(\alpha^L,\beta^L)=\pi_P(\sz^*\alpha,\,\sz^*\beta)=-\sz^* \pi_M(\alpha,\beta)
\end{align*}
and similarly for $\omega_P(\alpha^R,\beta^R)$. 
To see $\omega_P(\alpha^L,\beta^R)=0$, use \eqref{eq:canthink}, and the fact that 
$\pi_{M\times M^-}((0,\alpha),(\alpha,0))=0$.  

\end{proof}

As explained in Appendix \ref{app:groupoids}, the bracket relations among the vector fields $\alpha^L,\alpha^R$ 
give $P$ the structure of a local groupoid, where $\Lambda=\on{Gr}(\Mult_P)$ is obtained as the flow-out of  the triagonal 
$M_\Delta^{[2]}\subset {P\times P\times P}$ under 
vector fields of the form 
\begin{equation}\label{eq:2vf}
(\alpha^L,0,\alpha^L),\ \ \ (-\alpha^R,-\alpha^R,0),\ \ \alpha\in\Omega^1(M)
\end{equation}
\begin{proposition}
The submanifold $\Lambda$ is a Lagrangian submanifold (in the sense of symplectic geometry) of $P\times\ol{P\times P}$. 
\end{proposition}
\begin{proof}
Note that $\Lambda$ is generated from $M_\Delta^{[2]}$ already by the flow-outs of vector fields \eqref{eq:2vf} such that $\alpha=\d f$. Since these are Hamiltonian vector fields, they preserve the symplectic form. It therefore suffices to check that the tangent space 
to $\Lambda$ is Lagrangian along $M_\Delta$. But this is an easy consequence of 
Proposition \ref{prop:lr}. 
\end{proof}

\subsection{Notes}
As remarked in Weinstein's article \cite{wei:loc}, the problem of finding `symplectic realizations' is already present in Lie's work. The idea of a symplectic groupoid appeared in \cite{wei:sym}, and was developed in detail in \cite{cos:gro}. Independently, symplectic groupoids were discovered by Karasev \cite{kar:ana}, as the analogues of Lie groups for 
`non-linear Poisson brackets'. The original proofs of Karasev and Weinstein for the existence of a full symplectic realization start out by first constructing the realization locally, and then use uniqueness properties to patch these local definitions.  Assuming that one already has a (local) Lie groupoid $\G\rra M$ integrating the cotangent Lie algebroid $A=T^*M$, the question arises how to construct the multiplicative symplectic 2-form on the groupoid. In Mackenzie-Xu \cite{mac:int}, this is achieved using the general theory of integration of Lie bialgebroids. 
A more direct argument may be found in Bursztyn-Cabrera-Ortiz \cite{bur:lin}. 

Cattaneo-Felder \cite{cat:poi} gave a conceptual construction of canonical global `topological' symplectic groupoid, as the  phase space of the Poisson sigma model associated to the Poisson manifold. It may be regarded as a space of `Lie algebroid paths' in $T^*M$ modulo Lie algebroid homotopies. In general, the Cattaneo-Felder groupoid is smooth only on some neighborhood of the unit section; the precise criteria for global smoothness were determined in the work of Crainic-Fernandes \cite{cra:intpoi}. 

The finite-dimensional approach towards construction of a local symplectic groupoid, as discussed in these notes, is based on the papers of M\u{a}rcu\c{t}-Crainic \cite{mar:red} and Frejlich-M\u{a}rcu\c{t} \cite{fre:rea}. (In fact, the latter authors prove a more general Dirac-geometric version.) The M\u{a}rcu\c{t}-Crainic formula in Section \ref{subsec:cramar} has a simple interpretation in terms of the 
 Cattaneo-Felder picture: The Poisson spray determines an `exponential map' from an open neighborhood of the zero section of the cotangent Lie algebroid $A=T^* M$ to a neighborhood of the unit section of the Cattaneo-Felder groupoid, and the 2-form in 
 \ref{subsec:cramar} is simply the pullback of the symplectic form under this map. Details may be found in Broka-Xu \cite{bro:sym}, 
 where the construction is further generalized to holomorphic symplectic groupoids of holomorphic Poisson manifolds. 
 
Our construction of the local groupoid structure of $\G$ via left-and right-invariant vector fields seems to be new. Using the interpretation of Lie algebroids as linear Poisson manifolds, it also implies the local integrability of arbitrary Lie algebroids to local groupoids. In a recent article, Cabrera-M\u{a}rcu\c{t}-Salazar \cite{cab:con} 
give an explicit formula for the groupoid multiplication in terms of sprays, as well as 
formulas for the local integration of infinitesimally multiplicative structures to local Lie groupoids.

\section{Poisson Lie groups}\label{sec:poissonliegroups}

\subsection{Poisson Lie groups, Poisson actions}
\begin{definition}
A \emph{Poisson Lie group} is a Lie group $G$, together with a Poisson structure 
$\pi$, such that the multiplication map 
\[ \on{Mult}_G\colon G\times G\to G,\ (a,b)\mapsto ab\]
is a Poisson map. 
\end{definition}
\begin{example}
\begin{enumerate}
\item Any Lie group is a Poisson Lie group for the zero Poisson structure. 
\item Any vector space $V$, equipped with a linear Poisson structure, is a Poisson Lie group. Here the group multiplication is the addition map $V\times V\to V$. (Recall that this means that $V=\g^*$, for some Lie algebra $\g$, with Poisson bracket the Lie-Poisson structure.)  
\end{enumerate}
\end{example}
We will obtain more interesting examples once we have the classification theory. 
\begin{definition}\cite{se:dr}
An action of a Poisson Lie group $(G,\pi_G)$ on a Poisson manifold $(M,\pi_M)$ is called a \emph{Poisson action} if the action map $\A\colon G\times M\to M$ is Poisson.  
\end{definition}
Note that a Poisson action does not preserve the Poisson structure on $M$, in general.

The multiplicative property of Poisson Lie group structures $\pi$ has a simple interpretation in terms of Dirac geometry. Let $E=\on{Gr}(\pi)\subset \T G$
be its graph. Recall that the tangent bundle $TG$ of $G$ has the structure of a Lie group, 
while the cotangent bundle  is a Lie groupoid $T^*G\rra \g^*$. (See example \ref{ex:cotangentgroupoid}.) Taking a direct product, 
we obtain a Lie groupoid structure on $\T G$, 
\[ \T G\rra \g^*.\] 
By definition, if $x_i=v_i+\mu_i\in \T_{g_i}G$ and $x=v+\mu\in\T_g G$,
then $x=x_1\circ x_2$ if and only if $g=g_1g_2$ and 
\begin{equation}\label{eq:multrel1}
v=(\Mult_G)_*(v_1,v_2),\ \ \ (\mu_1,\mu_2)=(\Mult_G)^*\mu.
\end{equation} 
We see that $\Mult_G$ is Poisson if and only if for all $x\in E_g$, and given $g_1,g_2\in G$ with $g=g_1g_2$, there exist  $x_i\in E_{g_i},\ i=1,2$ with $x=x_1\circ x_2$. 
\begin{theorem}\label{th:poissonliegroupdirac}
A Poisson structure $\pi$ on a Lie group $G$ is a Poisson Lie group structure if and only if $\on{Gr}(\pi)=E$ is a Lie subgroupoid  of $\T G\rra \g^*$. 
\end{theorem}
\begin{proof}
Suppose $\Mult_G$ is a Poisson map. Given $y_i\in E_{g_i}$ with $\sz(y_1)=\tz(y_2)$,  we have to show that $y=y_1\circ y_2$ lies in $E_g$.  Let $x\in E_g$, 
and  $x_i\in E_{g_i}$ with $x=x_1\circ x_2$. Then
\[ \l x,\,y\r=\l x_1\circ x_2,\ y_1\circ y_2\r=\l x_1,y_1\r+\l x_2,y_2\r=0,\]
and since this holds for all $x$ it follows that $y\in E_g^\perp=E_g$. 

Conversely, suppose $E$ is a subgroupoid. Then $\sz\colon E\to \g^*$ is a fiberwise isomorphism.  Let $x\in E_g$, and $g_1,g_2\in G$  with $g=g_1g_2$. Let 
$x_2\in E_{g_2}$ be the unique element with $\sz(x_2)=\sz(x)$, and 
let $x_1=x\circ x_2^{-1}\in E_{g_1}$. Then $x=x_1\circ x_2$. 
\end{proof}

Poisson actions can be described similarly. Any $G$-action 
$\A\colon G\times M\to M$ 
on a manifold determines an equivariant map 
\[ \Phi\colon T^*M\to \g^*,\ \ \ \l \Phi(\mu),\xi\r=\l\mu,\xi_M\r.\]
In other words, $\Phi$ is the moment map for the cotangent lift of the $G$-action. 
The groupoid $T^*G\rra \g^*$ acts on $E=T^* M\to \g^*$, with 
\[ \mu\circ \nu=\nu'\ \ \ \ \Leftrightarrow\ \ \ \ 
(\mu,\nu)=(T_{(g,m)}\A)^*\, \nu' \]
for $\mu\in T_g^*G,\ \nu\in T_m^*M,\ \ \nu'\in T^*_{g.m}M$. (We refer to Appendix \ref{app:groupoids} for background on groupoid actions.)  In particular, the action of $\mu$ on $\nu$ is defined if and only of 
$\sz(\mu)=\Phi(\nu)$, and in this case $\Phi(\mu')=\tz(\mu)$. Together with the 
tangent map $T\A\colon TG\times TM\to TM$, one obtains a groupoid action 
of $\T G\rra \g^*$ on $\T M\to \g^*$.

\begin{proposition}\label{prop:poissonaction1}
Suppose $(G,\pi_G)$ is a Poisson Lie group, and $(M,\pi_M)$ a Poisson manifold. An action $G\times M\to M$ is Poisson if and only if the groupoid action of $\T G\rra g^*$ on $\T M\ra \g^*$ restricts to a groupoid action of $E=\on{Gr}(\pi_G)\rra \g^*$ on 
$\on{Gr}(\pi_M)\ra \g^*$. 
\end{proposition}

The proof is parallel to that of Theorem \ref{th:poissonliegroupdirac}, and is left as an exercise.

\subsection{Invariant functions and 1-forms for Poisson actions}\label{subsec:bullet}
Let $G$ be a Poisson Lie group. While a Poisson action of $G$ on a Poisson manifold $M$ need not preserve the Poisson structure on $M$, we have: 
\begin{proposition}\label{prop:poissoninvt}
 For any Poisson action, the space $C^\infty(M)^G$ of $G$-invariant functions is closed under the Poisson bracket. 
\end{proposition}
\begin{proof}
Let $\A\colon G\times M\to M$ be the action, and $\pr_2\colon G\times M\to M$ the  projection to the second factor. A function $f$ is $G$-invariant if the two pull-backs 
to $G\times M$ coincide: $ \A^*f=\pr_2^*f$.  Since $\A,\pr_2$ are both Poisson maps, this condition is preserved under Poisson brackets. 
\end{proof}
We will need the following related fact. Let $G$ be a Poisson Lie group with an action $\A\colon G\times M\to M$.
Using the structure of $E=\on{Gr}(\pi_G)$ as a subgroupoid of $\T G$, we obtain an action $\bullet$ of $G$ on $\T M$, 
\[ g\bullet y:=x\circ y\]
for $g\in G$ and $y\in \T M$, where $x$ is the \emph{unique} element in 
$E|_g$ such that $\sz(x)=\Phi(y)\in\g^*$.  It is easy to see that 
the $\bullet$-action preserves the subbundle $TM\subset \T M$, on which it coincides with the usual 
tangent lift of the $G$-action on $M$. The action also preserves the metric of $\T M$.
Finally, if $M$ has a Poisson structure for which the 
given action is Poisson, then the subbundle $E_M=\on{Gr}(\pi_M)$ is  invariant under the $\bullet$-action.
In contrast to the tangent/cotangent lift of the action, this $\bullet$-action on $\T M$ does not preserve the Courant bracket, in general. 
Nevertheless, we have:  
\begin{proposition}\label{prop:courinvt}
If two sections of $\T M$ are invariant under the $\bullet$-action, then so is their Courant bracket.   
\end{proposition} 
\begin{proof}  \cite{lib:cou}. 
Given $m'=g.m$ and $y'\in \T_{m'}M$, there are unique elements $x\in E|_g$ and $y\in  \T_m M$ such that 
$y'=x\circ y$; by definition, $y'=g\bullet y$.  Consequently, for any section $\sigma$ of $\T M$, there is a unique section $\wt{\sigma}$ of 
$E\subset \T M\subset \T(G\times M)$ such that $\wt{\sigma}\sim_\A \sigma$, 
and $\sigma$ is $\bullet$-invariant if and only if the $G\times \T M\subset \T(G\times M)$-component of 
$\wt{\sigma}$ is constant in the $G$-direction, and equal to 
$\sigma$. That is, 
\[ \wt{\sigma}=\sigma+\tau\]
where $\sigma\in \Gamma(\T M)$ is regarded as a section of $G\times \T M$ (constant in $G$-direction), and $\tau$ is a section of $E\times M\subset \T G\times M$ 
(usually non-constant in the $M$-direction). Given two sections $\sigma_i\in \Gamma(\T M)$, we have 
\begin{equation}\label{eq:thisone}\wt{\Cour{\sigma_1,\sigma_2}}=
 \Cour{\wt{\sigma}_1,\wt{\sigma}_2},\end{equation}
since Courant brackets of related sections are again related.  If $\sigma_i$ are $\bullet$-invariant, and writing $\wt{\sigma}_i=\sigma_i+\tau_i$, we thus have show that only the first term in 
\[  \Cour{\wt{\sigma}_1,\wt{\sigma}_2}=
\Cour{\sigma_1,\sigma_2}+
\Cour{\sigma_1,\tau_2}+\Cour{\tau_1,\sigma_2}
+\Cour{\tau_1,\tau_2}\]
contributes to the $G\times \T M$-component. Now, it is clear that the second and third take values in $\T G\times M$. As for the last term, note that any section of 
$E\times M$ can be written as a sum of sections of the form $f\phi$, 
where $\phi$ is a section of $E$ and $f$ a function on $G\times M$. Given two sections of this form, the definition of the Courant bracket gives 
(at first only using that $\phi_i\in \Gamma(\T G)$), 
\[ \Cour{f_1\,\phi_1,\ \ f_2\,\phi_2}=
f_1 \big(\a(\phi_1)f_2\big) \phi_2-f_2 \big(\a(\phi_2)f_1\big) \phi_1
+f_2\,\l\phi_1,\phi_2\r\ \d f_1. \]
The first two terms lie in $\Gamma(\T G\times M)$, as required. If $\phi_i\in \Gamma(E)$ we have $\l\phi_1,\phi_2\r=0$, hence the last is zero.
\end{proof}

In terms of the Lie bracket on 1-forms, this proposition has the following consequence. 
\begin{corollary}
Let $G\times M\to M$ be a Poisson Lie group action on a Poisson manifold $M$. Then the 
space of $G$-invariant 1-forms is a Lie subalgebra of the Lie algebra of 1-forms on $M$
(with bracket the Lie algebroid bracket of $T^*M$).  
\end{corollary}
\begin{remark}
Ping Xu has pointed out that (for connected $G$), this result is an immediate consequence of a Theorem of Jiang-Hua Lu \cite{lu:poi} that the action Lie algebroid $M\times \g$ and the  cotangent Lie algebroid $T^* M$ form a \emph{matched pair}. 
For $M=G$, the result is proved in \cite{wei:rem}.
\end{remark}

\subsection{Drinfeld's classification theorems
}
In what follows, a pseudo-Riemannian metric $\l\cdot,\cdot\r$ (i.e, non-degenerate symmetric bilinear form) on a vector space $V$ will be referred to as simply a \emph{metric}. A subspace $S\subset V$ is called \emph{Lagrangian} if $S=S^\perp$. 
 (Such subspaces exist if and only if the metric has split signature.) By a \emph{metrized Lie algebra}, 
we mean a Lie algebra $\dd$ together with an 
$\ad(\dd)$-invariant metric. 
\begin{definition}
A \emph{Manin triple} $(\dd,\g,\h)$ is a metrized Lie algebra $\dd$, together with two Lagrangian Lie subalgebras $\g,\ \h$ of $\dd$ such that 
\[ \dd=\g\oplus \h\]
as a vector space. Let $G$ be a Lie group integrating $\g$. Given an extension of the adjoint action $g\mapsto \Ad_g$ of $G$ on $\g$ 
to an action by Lie algebra automorphisms of $\dd$ preserving the metric, with infinitesimal action the given adjoint action $\ad$  
of $\g\subset \dd$, then $(\dd,\g,\h)$ is called a \emph{$G$-equivariant Manin triple}. 
\end{definition}

For a $G$-equivariant Manin triple, we will use the same notation $\Ad_g$ for the adjoint action of $g\in G$ on $\g$ and for its extension to $\dd$.  If $G$ is connected and simply connected, then the integration of a Manin triple to a $G$-equivariant Manin triple is automatic. 
\begin{remarks}
\begin{enumerate}
\item The bilinear form defines a duality between $\g$ and $\h$; hence we can identify $\h=\g^*$ or $\g=\h^*$ as vector spaces. 
By invariance of the metric,  the action $\Ad$ of $G$ on $\dd$ descends to the 
coadjoint action on $\dd/\g\cong \g^*$. Note however that the action $\Ad_g$ does not preserve 
$\h$ as a subspace of $\dd$, in general. 
\item 
For each Manin triple $(\dd,\g,\h)$ there is a \emph{dual} Manin triple $(\dd,\h,\g)$, obtained by switching the roles of $\g$ and $\h$. Hence, for any Poisson Lie group $G$ there will be a \emph{dual} (connected and simply connected) Poisson Lie group $H$ (often written as $G^*$). 
\end{enumerate}
\end{remarks}
Suppose $(G,\pi)$ is a Poisson Lie group. Since $\Mult_G$ is a Poisson map, its differential at the identity satisfies 
$(T_{(e,e)}\Mult_G)(\pi_e,\pi_e)=\pi_e$. Since $T_{(e,e)}\Mult_G$ is just addition $\g\times\g\to \g$, this implies that the Poisson bivector field vanishes at the group unit: 
\[ \pi_e=0.\] 
The linear approximation of $\pi$ at $e$ is a linear Poisson structure on $\g$, which under the correspondence \eqref{eq:correspondence1} is 
dual to a Lie bracket on $\h:=\g^*$. Let $\dd=\g\oplus \h$ (as a direct sum of vector spaces) with the metric  $\l\cdot,\cdot\r$ given by the pairing between $\g$ and $\g^*$. Drinfeld showed that there is a unique Lie bracket on $\dd$, in such a way 
that $(\dd,\g,\h)$ is a Manin triple, and he proved the following result (at least for $G$ simply connected):
\begin{theorem}[Drinfeld's classification of Poisson Lie groups] \label{th:drinfeld}
For any Lie group $G$, there is a 1-1 correspondence 
\begin{equation}\label{eq:correspondence3}
\Big\{\begin{tabular}{c} Poisson Lie group  \\structures on $G$ 
\end{tabular}\Big\}
\stackrel{1-1}{\longleftrightarrow} 
\Big\{\begin{tabular}{c} $G$-equivariant \\ Manin triples $(\dd,\g,\h)$ 
\end{tabular}\Big\}
\end{equation}
This correspondence takes a $G$-equivariant Manin triple $(\dd,\g,\h)$ to the bivector field $\pi\in\mf{X}^2(G)$ given by 
\begin{equation}\label{eq:bivectorG}
 \pi\big(\l\theta^L,\nu_1\r,\ \l\theta^L,\nu_2\r\big)\Big|_g=\Big\l \pr_\g(\Ad_g \nu_1),\ \pr_\h(\Ad_g\nu_2)\Big\r \end{equation}
 for $\nu_1,\nu_2\in \h$. 
\end{theorem}
On the left hand side of \eqref{eq:bivectorG}, $\theta^L\in \Omega^1(G,\g)$ are the left-invariant Maurer-Cartan forms on $G$; thus $\l\theta^L,\nu_i\r$ are the left-invariant 1-forms extending $\nu_i\in \h\cong\g^*$. On the right hand side of the formula, the 
 $\nu_i$ are regarded as elements of $\dd$, and $\pr_\g,\pr_\h$ are the projections from $\dd=\g\oplus \h$ to the two summands. Note that the  expression is skew-symmetric in $\nu_1,\nu_2$, as required, since for $\zeta_1,\zeta_2$ in a given Lagrangian subspace 
 $\mf{l}\subset \dd$ (here, $\mf{l}=\Ad_g(\h)$), 
 \[ \l\pr_\g(\zeta_1),\pr_\h(\zeta_2)\r+\l\pr_\h(\zeta_1),\pr_\h(\zeta_2)\r=\l\zeta_1,\zeta_2\r=0.\]
\begin{example}
\begin{enumerate}
\item For any Lie group $G$, with Lie algebra $\g$, one has the $G$-equivariant Manin triple 
\[ (\g\ltimes \g^*,\g,\g^*).\]
Here the semi-direct product is formed using the coadjoint representation of $\g$, and the bilinear form on $\g\ltimes\g^*$ is given by the pairing. The corresponding Poisson Lie group structure on $G$ is given by $\pi=0$. The dual Manin triple 
$(\g\ltimes \g^*,\g^*,\g)$ determines the Lie-Poisson structure on $\g^*$, regarded as a Lie group under addition. 
\item Let $K$ be a compact, connected Lie group, with complexification $G=K^\C$. Viewing $G$ as a real Lie group, it has an Iwasawa decomposition 
\[ G=KAN.\]
(If $K=\U(n)$, then $G=\GL(n,\C)$. The subgroup  $N$ consists of upper triangular matrices with $1$s on the diagonal, and $A$ consists of diagonal matrices with positive diagonal entries.) 
Let $(\cdot,\cdot)$ be an invariant inner product on $\k$, extend it to a $\C$-valued bilinear form on $\k^\C$, and let $\l\cdot,\cdot\r$ be its imaginary part. Then both
$\k$ and $\a\oplus \n$ are Lagrangian Lie subalgebras for this bilinear form, and 
\[ (\k^\C,\k,\a\oplus\n)\]
is a $K$-equivariant Manin triple. The resulting Poisson Lie group structure on $K$ is known as the \emph{Lu-Weinstein Poisson structure} or \emph{Bruhat Poisson structure} \cite{lu:poi,le:al}. 
The Poisson structure on $K^*=AN$, defined by the dual Manin triple, features in the \emph{Ginzburg-Weinstein isomorphism} \cite{gi:lp}; see \cite{al:po} for interesting aspects of this Poisson structure. 
\item 
The following example describes a \emph{complex} Poisson Lie group structure on a complex reductive Lie group $G$. Let  $\g$ be its Lie algebra, with some choice of a 
$G$-invariant non-degenerate symmetric bilinear form. Denote by $\ol{\g}$ the same Lie algebra with the opposite bilinear form, and embed $\g$ as the diagonal $\g_\Delta\subset \g\oplus \ol{\g}$. Fix 
a Cartan subalgebra $\h\subset \g$ and a system of positive roots, let $\n_\pm$ be the sum of positive/negative root spaces, and let  
\[ \mf{b}_+=\h\oplus \n_+,\ \ \mf{b}_-=\h\oplus \n_-\]
be the corresponding opposite Borel subalgebras. Both projections $\pr_\h\colon \mf{b}_\pm\to \h$ 
are Lie algebra homomorphisms. Let $\mathfrak{u}\subset \dd$ consist of pairs $(\xi_1,\xi_2)\in \mf{b}_+\oplus \mf{b}_-$ such that 
\[ \pr_\h(\xi_1)=-\pr_\h(\xi_2).\]
Then 
\[ (\g\oplus \ol{\g},\ \g_\Delta,\ \mathfrak{u})\] 
is a $G$-equivariant Manin triple, defining the so-called \emph{standard Poisson Lie group structure} on $G$.
\item
Using the same notation as in the previous example, let $\dd=\g\oplus\ol{\h}$ be the direct sum. Define inclusions 
\[ \mf{b}_\pm \hra  \dd=\g\oplus\ol{\h},\ \xi\mapsto (\xi,\pm \pr_\h(\xi)).\]
Then $(\dd,\mf{b}_+,\mf{b}_-)$ is a Manin triple. 
\item 
Let $(\dd,\g,\h)$ be a Manin triple. Then 
\[ (\dd\oplus \ol{\dd},\dd_\Delta,\g\oplus \h)\]
is again a Manin triple, called the \emph{Drinfeld double}. For instance, in the case of a compact Lie group $K$ with Lu-Weinstein Poisson structure, it follows that $G=K^\C$ also has a Poisson Lie group structure. 
\end{enumerate}
\end{example}
A \emph{Poisson homogeneous space} is a Poisson manifold $M$ together with a Poisson action of a Poisson Lie group $G$, such that $M$ is a homogeneous space for the underlying $G$-action. Given a closed subgroup $K\subset G$, we are interested in the Poisson structures $\pi_M$ on $M=G/K$ for which the $G$-action on $M$ is Poisson. 
\begin{theorem}[Drinfeld's classification of Poisson homogeneous spaces]
\label{th:drinhom}
Let $(G,\pi_G)$ be a Poisson Lie group, with equivariant Manin triple 
$(\dd,\g,\h)$. Let $K\subset G$ be a closed subgroup, and $M=G/K$. 
Then there is a 1-1 correspondence 
\begin{equation}\label{eq:correspondence4}
\Big\{\begin{tabular}{c} Poisson structures on $M$ such\\
 that the $G$-action  is Poisson
\end{tabular}\Big\}
\stackrel{1-1}{\longleftrightarrow} 
\Big\{\begin{tabular}{c} $K$-invariant Lagrangian Lie 
\\ subalgebras $\mf{l}\subset \dd$ with $\mf{l}\cap \g=\k$
\end{tabular}\Big\}
\end{equation}
\end{theorem}
In this theorem, neither $G$ nor $K$ are assumed to be connected or simply connected. The statement in Drinfeld's original paper \cite{dri:poi} is slightly less precise; the version given here can be found in \cite{rob:cla} as well as \cite{me:dir}. 

\subsection{Notes} The notion of a Poisson Lie group was introduced by Drinfeld \cite{dr:ha}. Poisson actions of Poisson Lie groups were first considered by Semenov-Tian-Shansky \cite{se:wh}, who used them to interpret examples of hidden symmetries in integrable systems. A distinguished class of Poisson Lie group structures are the \emph{quasi-triangular Poisson Lie groups}, defined by a classical \emph{$r$-matrix} satisfying a modified classical Yang-Baxter equation; for semi-simple complex Lie algebras these were classified by Belavin-Drinfeld \cite{be:tr}. Poisson Lie group structures on compact Lie groups, and their homogeneous spaces, were studied by  Lu-Weinstein \cite{lu:poi}, and independently by Levendorskii-Soibelman \cite{le:al}. The theory of 
Poisson Lie groups, and their homogeneous spaces, 
has been generalized to (suitably defined) categories of Dirac manifolds; see \cite{jot:dir}, \cite{ort:mu},\ \cite{lib:dir},\ \cite{rob:cla},\ \cite{me:dir}.  

Since the Poisson tensor of any Poisson Lie group $G$ vanishes at the group unit $e$, one can ask about \emph{linearizability} of $\pi$ near $e$. There are counter-examples showing that Poisson Lie groups are not always linearizable. For the dual Poisson Lie group $G^*$ of a compact Lie group $G$, the linearizability follows from the Ginzburg-Weinstein theorem. In \cite{al:lin}, it was shown that the dual of any \emph{coboundary} Poisson Lie group is linearizable at $e$. The proof uses Dirac geometry.

\section{The isomorphism $TG\oplus T^*G\cong G\times \dd$}
\label{sec:drinfeldtheorem}
We will now explain how to obtain Drinfeld's classification results for Poisson Lie groups and Poisson homogeneous spaces, using Dirac-geometric 
ideas. 

\subsection{Poisson Lie group structures}
Let $\pi_G$ be a Poisson Lie group structure on $G$, and $E=\on{Gr}(\pi_G)\subset \T G$ its graph. As explained in Section \ref{subsec:bullet}, given a $G$-action $g\mapsto \A_g$ on a manifold $M$ we obtain a lift to an action on $\T M$, denoted by a symbol $\bullet$.
This action comes from the groupoid action of $E\subset \T G\rra \g^*$ on $\T M\to \g^*$, and although it restricts to the 
standard tangent lift on $TM\subset \T M$, it is different from the standard lift $g\mapsto \T\A_g$ of the $G$-action.

\subsubsection{Left trivialization of $\T G$}
We will apply this to $M=G$ with the action $g\mapsto l_g$ given by left multiplication.  We will obtain the Manin triple $(\dd,\g,\h)$ from the triple 
$(\T G,\ TG,\ E)$ by  `left trivialization':
\[ \dd:=\Gamma(\T G)^{\bullet-inv},\ \ 
\g=\Gamma(TG)^{\bullet-inv},\ \ 
\h=\Gamma(E)^{\bullet-inv}.
\] 
where the superscript indicates sections that are invariant under the $\bullet$-action. 
Since the $\bullet$-action preserves the metric, it follows that 
$\dd$ is a metrized vector space, with $\g$ and $\h$ as complementary Lagrangian subspaces. By 
Proposition \ref{prop:courinvt}, the Courant bracket on $\Gamma(\T G)$ restricts to a bracket on invariant 
sections, hence we obtain a bracket on $\dd$. We claim that this bracket is skew-symmetric, and hence is a Lie bracket. To see this, let $\sigma_1,\sigma_2$ be two 
invariant sections. By property \eqref{eq:iii} of the Courant bracket,  we have that  
\[ \Cour{\sigma_1,\sigma_2}+\Cour{\sigma_2,\sigma_1}=
\d \l \sigma_1,\sigma_2\r.\] 
Since the $\bullet$-action preserves the metric, the function $\l\sigma_1,\sigma_2\r\in C^\infty(G)$ is $G$-invariant, and hence is \emph{constant}. It follows that $\d \l \sigma_1,\sigma_2\r=0$, 
proving skew-symmetry of the bracket on $\dd$.
For this Lie bracket, the metric on $\dd$ is $\ad$-invariant (by the property \eqref{eq:ii} of the Courant bracket).
Furthermore, $\h,\g\subset\dd$ are Lagrangian Lie subalgebras since $E,\ TG\subset \T G$ are Dirac structures. (Of course, the Lie bracket on $\g$ is the given one.) Consider now the second $G$-action on $\T G$, 
obtained as the \emph{standard} lift $g\mapsto \T r_{g^{-1}}$ of the action by \emph{right multiplication}. 
This $G$-action commutes with the action $\bullet$, and (unlike the action $\bullet$) preserves not only the metric but also the Courant bracket. Hence it defines an action $g\mapsto \Ad_g$ on $\dd$ by metric-preserving Lie algebra automorphisms .
Note that this action reduces to the usual adjoint action on $\g$ (as left-invariant vector fields). We conclude that $(\dd,\g,\h)$ is a $G$-equivariant Manin triple. 

We have thus constructed a \emph{left trivialization}
\begin{equation}\label{eq:lefttrivial}
\T G=G\times \dd,\ \ TG=G\times\g,\ \ E=G\times \h\end{equation}
in such a way that the Courant bracket on $\Gamma(\T G)$ is the unique extension of the Lie bracket on $\dd$ (regarded as constant sections) to a bilinear form on $\Gamma(\T G)$, satisfying the properties \eqref{eq:iii} and \eqref{eq:leibnitzc} of the Courant bracket. 
The $G$-action $g\mapsto \T r_{g^{-1}}$ on $\T G$ translates into an action by Courant automorphisms, $g.(a,\zeta)=(a g^{-1},\ \Ad_g(\zeta))$. 

\subsubsection{The dressing action}
Since the anchor $\a\colon \T G\to TG$ is just the obvious projection along $T^*G$, it is bracket preserving. Its restriction to $\bullet$-invariant sections defines a Lie algebra action
\[ \varrho\colon \dd\to \mf{X}(G),\]
with $\a(g,\zeta)=\varrho(\zeta)_g$, 
extending the $\g$-action $\xi\mapsto \xi^L$. This action (or its restriction to $\h\subset \dd$) is  known as the 
\emph{dressing action}. The following proposition gives an explicit formula.
\begin{proposition}
The dressing action of $\dd$ on $G$ is given in terms of the contraction with left-invariant vector fields by 
\begin{equation}\label{eq:dressingaction}
 \iota_{\varrho(\zeta)}\theta^L=\Ad_{g^{-1}}\pr_\g(\Ad_g\zeta).\end{equation}
\end{proposition}
\begin{proof}
Note that the anchor map is equivariant for the standard lift $\T r_{g^{-1}}$ of the right action. Hence, in terms of the 
identification $\T G\cong G\times \dd$,  
\[ \a(g,\zeta)=(r_{g})_*\a(e,\Ad_g\zeta)=(l_g)_*\Ad_{g^{-1}} \a(e,\Ad_g\zeta).\]
At the group unit, the anchor is simply the projection $\pr_\g\colon\dd=\T_e G\to \g$.
Hence $\a(e,\Ad_g\zeta)=\pr_\g(\Ad_g\zeta)$. 
\end{proof}
We can use the dressing action $\varrho$ to describe the inverse map to the isomorphism $\T G\cong G\times \dd$.
\begin{proposition}
The inverse map to the left trivialization \eqref{eq:lefttrivial} is given on the level of sections by  
\[ \mathsf{e}\colon \dd\to \Gamma(\T G),\ \ \zeta\mapsto \mathsf{e}(\zeta)=\varrho(\zeta)+\l\theta^L,\zeta\r.\]
Here $\theta^L\in \Omega^1(G,\g)$ is the left-invariant Maurer-Cartan form on $G$, and $\varrho(\zeta)\in\mf{X}(G)$ 
are the vector fields given by the formula \eqref{eq:dressingaction}. 
\end{proposition}
\begin{proof}
We know that the left trivialization $\T G\cong G\times \dd$ restricts to the usual left trivialization $TG\cong G\times \g$.
Hence, the inverse isomorphism is given on $\g\subset \dd$ by $\xi\mapsto \xi^L$. For more general elements 
$\zeta\in\dd$, since the tangent part of the isomorphism $G\times\dd\cong \T G$ is given by the anchor, 
the image of $\zeta$ is of the form $\varrho(\zeta)+\alpha$ for some 1-form $\alpha$.
Since the isomorphism preserves metrics, we have for $\zeta\in\dd$ and $\xi\in\g$, 
\[ \iota(\xi^L)\l\theta^L,\zeta\r=\l\xi, \zeta\r=
\l \mathsf{e}(\xi),\ \mathsf{e}(\zeta)\r=
\l  \xi^L,\ \varrho(\zeta)+\alpha\r=\iota(\xi^L)\alpha.\]
This shows that $\alpha=\l\theta^L,\zeta\r$. \end{proof}

For $\nu\in \h=\g^*$, we have $\mathsf{e}(\nu)\in\Gamma(\on{Gr}(\pi))$, hence the formula shows that 
$\pi^\sharp\big(\l\theta^L,\nu\r\big)=\varrho(\nu)$. 
Consequently, if $\nu_1,\nu_2\in \g^*=\h$, then 
\[ \pi\big(\l\theta^L,\nu_1\r,\ \l\theta^L,\nu_2\r\big)=\big\l \varrho(\nu_1),\ \l\theta^L,\nu_2\r\big\r
=\l \pr_\g(\Ad_g \nu_1),\ \Ad_g\nu_2\r,\]
which is the formula \eqref{eq:bivectorG} for the Poisson tensor on $G$.

\subsubsection{The groupoid structure}
We next describe the groupoid structure of $\T G\rra \g^*$ in terms of the trivialization  \eqref{eq:lefttrivial}. 
\begin{proposition}\label{prop:groupoidstr}
The groupoid structure $G\times \dd\rra \h$, induced by the groupoid structure on 
$\T G\rra \g^*$ by the isomorphism \eqref{eq:lefttrivial}, has source and 
target maps 
\[ \sz(g,\zeta)=\pr_\h(\zeta),\ \ \tz(g,\zeta)=\pr_\h(\Ad_g \zeta),\]
and $(g,\zeta)=(g_1,\zeta_1)\circ (g_2,\zeta_2)$ if and only if 
\[ g=g_1g_2,\ \ \ \ \ \zeta=\zeta_2+\Ad_{g_2^{-1}} \big(\pr_\g(\zeta_1)\big)\]

\end{proposition}
\begin{proof}
At the group unit, the source and target map of $\T_e G=\g\oplus \g^*\rra \g^*$ are simply the projections to the second summand $\g^*\cong \h$. For general $(g,\zeta)
=x\circ \zeta$, with the unique $x\in E|_g$ for which $\sz(x)=\tz(\zeta)$, we 
obtain 
\[\sz(g,\zeta)=\sz(x\circ \zeta)=\sz(\zeta)=\pr_\h(\zeta).\] 
On the other hand, since the map $\tz\colon \T G\to \g^*$ is invariant under $\T r_{g^{-1}}$, 
\[ \tz(g,\zeta)=\tz( \T r_{g^{-1}}(g,\zeta))=\tz(e,\Ad_g\zeta)=\pr_\h(\Ad_g\zeta).\]
To describe the groupoid multiplication, consider the action $\ast$ of $G$ on $\dd=\g\oplus \h=\g\oplus \g^*$, 
given as the direct sum of the adjoint action on $\g$ and the coadjoint action on $\g^*$.  Under the identification $\dd=\T_e G$, we have  
\[ g\ast \zeta=x\circ \zeta\circ x^{-1}\]
where $x\in E_g$ is the unique element such that
$\sz(x)=\pr_\h(\zeta)$. In contrast to the action $\Ad_g$ on $\dd$, this action $\ast$ preserves the direct sum 
decomposition, but not the Lie bracket. We have that 
\[ \pr_\g(g\ast \zeta)=\Ad_g\pr_\g(\zeta)\]
for all $g\in G$ and all $\zeta\in\dd$. The groupoid multiplication of $\T G$ is given on the subgroupoid $\T_e G=\g\oplus \g^*\rra \g^*$ by addition in the first summand, and the trivial groupoid structure in the second summand. This may be written $\zeta_1\circ \zeta_2=\zeta_2+\pr_\g(\zeta_1)$. For general composable elements $(g_1,\zeta_1),\ (g_2,\zeta_2)$,  let $x_i\in E|_{g_i}$ be the unique elements such that $\sz(x_i)=\tz(\zeta_i)=\pr_\h(\zeta_i)$. 
Then $(g_i,\zeta_i)=x_i\circ \zeta_i$, and 
\[ (g_1,\zeta_1)\circ (g_2,\zeta_2)=x_1\circ \zeta_1\circ x_2\circ \zeta_2=
x_1\circ x_2\circ \big( (g_2^{-1}\ast \zeta_1)\circ \zeta_2\big)
=\big(g_1g_2,\  (g_2^{-1}\ast \zeta_1)\circ \zeta_2\big).
\]
But 
$
(g_2^{-1}\ast \zeta_1)\circ \zeta_2=
\zeta_2+\pr_\g (g_2^{-1}\ast \zeta_1)
=\zeta_2+\Ad_{g_2^{-1}}\pr_\g( \zeta_1)$. 
\end{proof}

\subsubsection{Poisson Lie group structures from Manin triples}
Let us now assume conversely that we are given a $G$-equivariant Manin triple $(\dd,\g,\h)$.
Define a map 
\[ \mathsf{e}\colon \dd\to \Gamma(\T G),\ \ \mathsf{e}(\zeta)=\varrho(\zeta)+\l\theta^L,\zeta\r\]
where the vector fields $\varrho(\zeta)$ are given in terms of their contractions with the Maurer-Cartan forms by 
$ \iota_{\varrho(\zeta)}\theta^L=\Ad_{g^{-1}}\pr_\g(\Ad_g \zeta)$. Note that $\varrho(\xi)=\xi^L$ for $\xi\in\g$. 
\begin{lemma}\label{lem:technical}
The map $\mathsf{e}$ preserves metrics and brackets: 
\[ \l \mathsf{e}(\zeta_1),\mathsf{e}(\zeta_2)\r
=\l \zeta_1,\zeta_2\r,\ \ \ \ 
\Cour{\mathsf{e}(\zeta_1),\mathsf{e}(\zeta_2)}
=\mathsf{e}([\zeta_1,\zeta_2]).\] In particular, 
$\varrho$ defines a Lie algebra action of $\dd$ on $G$.
\end{lemma}
This is verified by a straightforward, but un-illuminating calculation, see Appendix \ref{app:tech}.
Using the bundle isomorphism $G\times \dd\cong \T G$ defined by $\mathsf{e}$, we conclude that the Courant bracket 
on $\Gamma(\T G)$ is identified with the unique bracket on $\Gamma(G\times\dd)$ which extends the Lie bracket 
on $\dd$, in such a way that the properties \eqref{eq:iii} and \eqref{eq:leibnitzc} of the Courant bracket hold.
The image of $G\times \g$ under this isomorphism is $TG$. The image of $G\times\h$ is a Dirac structure $E\subset \T G$ transverse 
to $TG$; it is thus the graph of a Poisson tensor $\pi_G$. 

One easily checks that the isomorphism $\T G\cong G\times\dd$ takes the groupoid structure on $\T G\rra \g^*$ to the groupoid 
structure $G\times\dd\rra \h$, as described in Proposition \ref{prop:groupoidstr}. Since $G\times\h\rra \h$ is a subgroupoid of 
$G\times\dd$, it follows that $E=\on{Gr}(\pi_G)$ is a subgroupoid of $\T G$. That is, $\pi_G$ is a Poisson Lie group structure.

\subsection{Poisson homogeneous spaces}
We will need the following addendum to our discussion of pull-backs of Dirac structures, in Section \ref{subsec:pullbackdirac}.
\begin{lemma}\label{lem:pushforward}
Let $K$ be a Lie group, and $p\colon Q\to M$ be a principal $K$-bundle. Then there is a 1-1 correspondence between 
$K$-invariant Dirac structures $L_Q\subset \T Q$ with $\ker(T p)\subset L_Q$, and 
Dirac structures $L_M\subset \T M$. Under this correspondence, 
$L_Q=p^! L_M$, while
\[ L_M=p_! L_Q=\{y\in \T M|\ \exists z\in L_Q|\ z\sim_p y\}.\]
The Dirac structure $L_M$ is the graph of a Poisson structure if and only if $L_Q\cap TQ=\ker(T p)$. 
\end{lemma}
\begin{proof}
$p^!L_M$ consists of all $z\in \T Q$ such that $z\sim_p y$ for some $y\in L_M$, where $\sim_p$ is defines as in Section \ref{subsec:pullbackdirac}. Note that the set of elements $z\in \T Q$ satisfying $z\sim_p 0$ is exactly $\ker (T p)$. 
Hence, $p^!L_M$ is a $K$-invariant Dirac structure containing $\ker(T p)$. Conversely, if $L_Q$ is a Dirac structure with these properties, then its push-forward $L_M=p_! L_Q$ is a well-defined Dirac structure, with $p^!L_M=L_Q$. 
(The $K$-invariance guarantees that the fiberwise push-forward $(L_M)_m=p_! (L_Q)_q$ does not depend on the choice of $q\in p^{-1}(m)$. The condition $\ker(T p)\subset L_Q$ guarantees that $L_M$ is a smooth subbundle.) 
\end{proof}
\begin{remark}\label{rem:redu1}
One can also think of $L_M$ as a reduction $L_Q\qu K\subset \T Q\qu K=\T M$, 
\[ L_Q\qu K=\big(L_Q/\ker(Tp)\big)/K\subset \T Q\qu K=\big((\ker T p)^\perp/\ker Tp\big)/K.\]
\end{remark}

Suppose now that $G$ is a Poisson Lie group. Given a 
$G$-manifold $M$, consider once again the action of $G$ on $\T M$, defined by 
the action of the subgroupoid $E=\on{Gr}(\pi_G)$.
\begin{lemma}\label{lem:bulletequiv}
Let $\varphi\colon N\to M$ be a $G$-equivariant smooth map, and let $z\in \T N,\ y\in \T M$. 
Then 
\[ z\sim_\varphi y\ \ \ \Rightarrow\ \ \ g\bullet z\sim_\varphi g\bullet y\]
for all $g\in G$.
\end{lemma}
\begin{proof} The $G$-equivariance implies that the images of $z,\,y$ under the moment 
maps $\T N\to \g^*,\ \ \T M\to \g^*$ for the groupoid actions of $\T G\rra \g^*$ coincide. 
Hence, given $x\in \T G$, one has that $y'=x\circ y$ is defined if and only if 
$z'=x\circ z$ is defined. By definition of $\sim_\varphi$, one has $z'\sim_\varphi y'$ in this case. Taking $x\in E_g$, this shows that 
$g\bullet z\sim_\varphi g\bullet y$. 
\end{proof}
With these preparations we can give the classification of Poisson homogeneous spaces. 
\begin{proof}[Proof of Proposition \ref{th:drinhom}]
Let $M=G/K$, with $p\colon G\to G/K$ the 
quotient map. By Lemma \ref{lem:pushforward} the Dirac structures $L_M\subset \T M$ are in 1-1 correspondence 
with $K$-invariant Dirac structures $L_G\subset \T G$ containing $\ker(T p)$, and $L_M$ is the graph of a 
Poisson structure $\pi_M$ if and only if the intersection $L_G\cap TG$ is exactly $\ker(Tp)$. On the other hand, 
the $G$-action on $M$ will be a Poisson Lie group action if and only if $L_M$ is invariant under the $\bullet$-action; 
by Lemma \ref{lem:bulletequiv} this is the case if and only of $L_G$ is invariant under the $\bullet$-action. 

It follows that the Poisson structures $\pi_M$ making $M$ into a Poisson homogeneous space, are classified by 
the $G\times K$-invariant Dirac structures $L_G\subset \T G$ with the property $L_G\cap TG=\ker(T p)$. Here the 
$G$-action is the $\bullet$-action (lifting $g\mapsto l_g$), while the $K$-action is $k\mapsto \T r_{k^{-1}}$ (lifting 
$k\mapsto r_{k^{-1}}$. Now let $(\dd,\g,\h)$ be the $G$-equivariant Manin triple corresponding to the Poisson Lie 
group structure on $G$, and consider the left trivialization $\T G=G\times \dd$
defined by the $\bullet$-action. In terms of the trivialization, the $G\times K$-action is $(g,k).(a,\zeta)=
(gak^{-1},\Ad_k\zeta)$. Hence, the $G\times K$-invariant subbundles are of the form $L_G=G\times \mf{l}$ 
such that $\mf{l}\subset \dd$ is a $K$-invariant subspace. Such a subbundle is Lagrangian if and only if 
$\mf{l}$ is Lagrangian in $\dd$, and is a Dirac structure if and only if $\mf{l}$ is a Lagrangian Lie subalgebra. 
Since $\ker(Tp)\subset TG$ is $G\times\k\subset G\times\g$ in the trivialization, we see that $L_G\cap TG=\ker(Tp)$ if and only if 
$\mf{l}\cap \g=\k$. 
\end{proof}

\begin{remark}
By Remark \ref{rem:redu1}, we can think of $L_M=\on{Gr}(\pi)\subset \T M$ as a reduction $L_G\qu K$. 
In terms of the trivialization $\T G=G\times\dd$ we have 
$\ker(T p)^\perp=G\times \k^\perp$ and $\ker(T p)=G\times \k$
hence 
\[ L_M=G\times_K (\mf{l}/\k)\ \subset\  \T M=G\times_K (\k^\perp/\k).\] 
See \cite{rob:cla,me:dir} for further details. 
\end{remark}


\subsection{Notes} 
The Dirac-geometric approach to Drinfeld's classification theorems may be found in the paper of Liu-Weinstein-Xu \cite{liu:dist}, with a generalization to  Poisson homogeneous spaces for Poisson Lie group\emph{oid}s. The exposition given here incorporates ideas from \cite{lib:dir}, \cite{rob:cla} and \cite{me:dir}. The fact that the cotangent Lie algebroid $T^* M$ of a Poisson homogeneous space 
$M=G/K$ is a reduction $(G\times\mf{l})\qu K$
of the action Lie algebroid  (where $\mf{l}$ acts on $G$ by the dressing action)
was first observed by Jiang-Hua Lu \cite[Theorem 4.7]{lu:note}. 
\vskip.6in

\begin{appendix}
\vskip1.2in
\section{Sign conventions}\label{app:signs}
\subsection{Lie groups, Lie algebras}
For a Lie group $G$, with $\g=T_eG$, we denote by $X^L$ the left-invariant vector field 
and by $X^R$ the right-invariant vector field defined by $X\in\g$. That is, 
\[ X=X^L|_e=X^R|_e.\]
Similarly, we denote by $\theta^L\in \Om^1(G,\g)$  the left-invariant Maurer-Cartan form, and by $\theta^R\in \Omega^1(G,\g)$ the right-invariant Maurer-Cartan form. Thus \[ \iota(X^L)\theta^L=X=\iota(X^R)\theta^R.\]
For a matrix Lie group $G\subset \on{Mat}_\R(n)$, the Maurer-Cartan forms are given by 
\[ \theta^L=g^{-1} \d g,\ \ \ \theta^R=\d g\  g^{-1},\]
respectively, and they satisfy the equations 
\begin{equation}\label{eq:Maurer} 
\d \theta^L+\hh [\theta^L,\theta^L]=0,\ \ \ \d\theta^R-\hh [\theta^R,\theta^R]=0
\end{equation}
where $[\cdot,\cdot]$ is the commutator of matrices. Equation \eqref{eq:Maurer} implies, by a simple calculation using the Cartan calculus, that 
\begin{equation}\label{eq:leftright} [X^L,Y^L]=[X,Y]^L,\ \ [X^L,Y^R]=0,\ \ [X^R,Y^R]=-[X,Y]^R
\end{equation}
for all $X,Y$, where the brackets are commutators of matrices on the right hand sides and commutators of derivations of $C^\infty(M)$ on the left hand side. 
For a general Lie group, we use 
the equation $[X^L,Y^L]=[X,Y]^L$ to \emph{define} the Lie bracket; in this way the 
 Lie bracket for matrix Lie groups is the commutator. For Lie groupoids, we similarly define the corresponding Lie algebroid using left-invariant vector fields
(tangent to the $\tz$-fibers).

\subsection{Flows of vector fields}
We think of the Lie algebra of vector fields $\mf{X}(M)$ as the Lie algebra of the group of diffeomorphisms $\on{Diff}(M)$. The former act on $C^\infty(M)$ by Lie derivative, the latter by push-forward $\Phi_*=(\Phi^{-1})^*,\ f\mapsto f\circ \Phi^{-1}$. 
We define the flow $\Phi_t$ of a (complete) vector field $X$ 
 to be 
that 1-parameter group of diffeomorphisms such that 
\[ \ca{L}_X  =\f{\p}{\p t}|_{t=0} (\Phi_t)_*\]
as operators on functions.  For instance, the vector field $X=\f{\p}{\p x}$ in the real line has flow $\Phi_t(x)=x-t$, corresponding to a differential equation $\f{d x}{ d t}=-1$. 
More generally, the flow of  a time dependent vector field $X_t$ 
is defined by 
\[ \f{\p}{\p t}(\Phi_t)_*= (\Phi_t)_*\circ \L_{X_t},\]
consistent with `left trivialization'. 

\subsection{Poisson brackets for symplectic structures}
Consider $\R^2$ with coordinates $q$ (positions) and $p$ (momenta). The standard definition of the Poisson bracket in physics is such that $\{q,p\}=1$, while the standard 
convention for the symplectic form seems to be $\omega=\d q\wedge \d p$. For the Poisson bivector, we use the convention $\pi(\d f,\ \d g)=\{f,g\}$, which gives 
$\pi=\f{\p}{\p q}\wedge \f{\p}{\p p}$. With these conventions, the map 
$\pi^\sharp\colon T^*M\to TM$ inverts \emph{minus} the map $\omega^\flat\colon TM\to T^*M$, 
\[ \pi^\sharp\circ \omega^\flat=-\on{id}.\]
We use this in general for the definition of the Poisson structure associated to a symplectic structure.

\vskip.8in\section{Lie groupoids}\label{app:groupoids}
For a more detailed discussion of Lie groupoids, we refer to the books \cite{mac:gen} or 
\cite{moe:fol}.

\subsection{Groupoids and their actions}
A \emph{groupoid} $\G\rra M$ involves a set $\G$ of \emph{arrows}, a subset $i\colon M\hra \G$ of \emph{units} (or \emph{objects}), and two retractions  $\sz,\tz\colon \G\to M$ called \emph{source} and \emph{target}. (Retraction means $\tz\circ i=\sz\circ i=\on{id}_M$.)  
One thinks of $g$ as an arrow from its source $\sz(g)$ to its target $\tz(g)$. 
%
%
The groupoid structure is given by a  multiplication map, 
\[ \Mult_\G\colon \G^{(2)}\to \G,\ \ \ (g_1,g_2)\mapsto g_1\circ g_2,\]
where $\G^{(2)}=\{(g_1,g_2)\in \G^2|\ \sz(g_1)=\tz(g_2)\}$, and such that 
$\sz(g_1\circ g_2)=\sz(g_2),\ \ \tz(g_1\circ g_2)=\tz(g_1)$. It is thought of as a 
concatenation of arrows.
%
%
These are required to satisfy the following axioms: \vskip.2in

\begin{itemize}
\item[] {\bf 1. Associativity:} $(g_1\circ g_2)\circ g_3=g_1\circ (g_2\circ g_3)$ whenever $\sz(g_1)=\tz(g_2),\ \sz(g_2)=\tz(g_3)$.
\item[] {\bf 2. Units:} $\tz(g)\circ g=g=g\circ \sz(g)$ 
for all $g\in G$.
 \item[]{\bf 3. Inverses:} Every $g\in G$ is invertible: There exists 
 $h\in G$ such that $\sz(h)=\tz(g),\ \tz(h)=\sz(g)$, and such that 
 $g\circ h,\ h\circ g$  are units. 
\end{itemize}
\vskip.2in

The inverse of an element is necessarily unique, denoted $g^{-1}$, and  we have that 
$g\circ g^{-1}=\tz(g),\ \ g^{-1}\circ g=\sz(g)$. One denotes by 
\[ \G^{(k)}=\{(g_1,\ldots,g_k)\in \G^k|\ \sz(g_i)=\tz(g_{i+1})\}\]
the set of \emph{$k$-arrows}, with the convention $\G^{(0)}=M$. 

Note that the entire groupoid structure is encoded in the graph of the multiplication map, 
\[ \on{Gr}(\Mult_\G)=\{(g_1\circ g_2,g_1,g_2)\in \G^3|\ (g_1,g_2)\in \G^{(2)}\}\]
Thus, the set of all $(g,g_1,g_2)$ such that $g=g_1\circ g_2$ determines all the structure maps. For instance, the units are recovered as those elements $m$ such that $m=m\circ m$, while $\sz(g)$ is the unique unit $m$ such that $g\circ \sz(g)$ is defined. 

\begin{example}
A group $G$ is the same as a groupoid with a unique unit, $G\rra \pt$. At the opposite extreme, every set $M$ can be regarded as a trivial groupoid $M\rra M$ where all elements are units. 
\end{example}
\begin{example}
For any set $M$, one has the \emph{pair groupoid} $\on{Pair}(M)=M\times M\rra M$, with
\[ (m',m)= (m_1',m_1)\circ (m_2',m_2)\ \ \Leftrightarrow\ \ 
m_1'=m',\ m_1=m_2',\ m_2=m.\]
Here the units are given by the diagonal embedding $M\hra M\times M$, and 
the source and target of $(m',m)$ are $m$ and $m'$, respectively. 
\end{example}
\begin{example}
Given an action of a group $G$ on $M$, one has the \emph{action groupoid} 
$\G\rra M$. It may be defined as the subgroupoid of 
$G\times \on{Pair}(M)\rra M$, 
consisting of all $(g,m',m)\in G\times (M\times M)$ such that $m'=g.m$. Using the projection 
$(g,m',m)\mapsto (g,m)$ to identify $\G\cong G\times M$, the product reads as 
\[ (g,m)=(g_1,m_1)\circ (g_2,m_2) \ \ \Leftrightarrow\ \ 
g=g_1g_2,\ m=m_2,\ m_1=g_2. m_2.\]
\end{example}
Generalizing group actions, one can also consider \emph{groupoid actions} 
of a groupoid $\G\rra M$ on a set $P$ with a given map $\Phi\colon P\to M$. Such an action is given by 
a map 
\[ \G\ _{\sz}\times_\Phi P\to P,\ (g,p)\mapsto g\circ p\] 
such that $(g_1\circ g_2)\circ p=g_1\circ (g_2\circ p)$ 
whenever the compositions are defined, and $\Phi(p)\circ p=p$.
Sometimes we use subscripts for the action, to avoid confusion with the groupoid multiplication. For instance, $\G\rra M$ acts on itself by left multiplication 
$g\circ_L a=l_g(a)=g\circ a$ (here $\Phi=\tz$), and by right multiplication 
$g\circ_R a=r_{g^{-1}}(a)=a\circ g^{-1}$ (here $\Phi=\sz$), and it also acts 
on its units, by $g\circ_M m=\tz(g)$
(here $\Phi=\on{id}_M$). 

\subsection{Lie groupoids}
In the definition of a  \emph{Lie groupoid}, one requires in addition that the space of arrows $\G$ is a manifold, with $M$ an embedded submanifold, that $\sz,\tz$ are smooth submersions, and  $\on{Mult}_\G$ is smooth. This implies that the inversion map $\on{Inv}_G\colon g\mapsto g^{-1}$ is smooth, that $\G^{(k)}$ is an embedded submanifold of $\G^k$ for all $k$, and hence 
that $\on{Gr}(\Mult_\G)$ is a submanifold of $\G^3$. For instance, given a smooth action of a Lie group, the corresponding action groupoid is a Lie groupoid. 

Given a Lie groupoid $\G\rra M$, the left multiplication by a groupoid element 
$g\in \G$ with source $\sz(g)=m$ and target $\tz(g)=m'$ is a diffeomorphism 
\[ l_g\colon \tz^{-1}(m)\to \tz^{-1}(m'),\ \ \ \ a\mapsto g\circ a\]
A vector field $X\in \mf{X}(\G)$ is called \emph{left-invariant} if it is tangent to the $\tz$-fibers and for all  $g\in \G$, the restrictions of $X$ to the corresponding 
$\tz$-fibers of $m=\sz(g),\ m'=\tz(g)$ are related by $l_g$. The left-invariant vector fields form a Lie subalgebra $\mf{X}(\G)^L$ of the Lie algebra of vector fields on $\G$. 
Similarly a vector field is called \emph{right-invariant} if it is tangent to all $\sz$-fibers, and  
for all $g$ as above the restriction to the $\sz$-fibers of $m',m$ are related by 
$r_{g^{-1}}$. The right-invariant vector fields form a Lie subalgebra
$\mf{X}(\G)^R$. 

\subsection{The Lie algebroid of a Lie groupoid}
Let 
\[ A=\nu(\G,M)\] 
be the normal bundle of $M$ in $\G$. The difference $T\tz-T\sz\colon T\G\to TM$ vanishes on $TM\subset T\G$, hence descends to a bundle map 
$\a\colon A\to TM$ called the \emph{anchor map}. For $\sigma\in \Gamma(A)$, there are unique vector fields
\[ \sigma^L\in\mf{X}(\G)^L,\ \ \ \sigma^R\in\mf{X}(\G)^R\]
mapping to $\sigma$ under restriction to $T\G|_M$ followed by the quotient map. 
We have that 
\[ \a(\sigma)\sim_i\ \sigma^L-\sigma^R,\ \ \ \sigma^L\sim_{\sz}\a(\sigma),\ \ \ \sigma^R\sim_{\tz} -\a(\sigma).\] 

We will use the identification of $\Gamma(A)$ with \emph{left-invariant} vector fields on $\G$ to define a Lie bracket on $\Gamma(A)$. 
(This choice is dictated by the conventions for Lie groups, as discussed in Appendix \ref{app:signs}.) With these data, $(A,\a,[\cdot,\cdot])$ is a Lie algebroid, called the Lie algebroid of the Lie groupoid $\G$. Conversely, one refers to $\G$ as an \emph{integration of $A$}. Not every Lie algebroid admits an integration; the precise obstructions to integrability were determined by Crainic-Fernandes \cite{cra:intlie}.
On the other hand, one always has an integration to a local Lie groupoid. 

Essentially by construction, the left and right invariant vector fields satisfy the bracket relations, for $\sigma,\tau\in\Gamma(A)$: 
\begin{equation}\label{eq:bracketrelations}
 [\sigma^L,\tau^L]=[\sigma,\tau]^L,\ \ \ 
[\sigma^L,\tau^R]=0,\ \ \ [\sigma^R,\tau^R]=-[\sigma,\tau]^R.\end{equation}
For $i=0,1,2$ and $\sigma\in \Gamma(A)$, the vector fields on $\G\times \G\times \G$, 
\[ X_\sigma^0=(-\sigma^R,-\sigma^R,0),\ \ X_\sigma^1=(0,\sigma^L,-\sigma^R),\ \ 
X_\sigma^2=(\sigma^L,0,\sigma^L).\]
are all tangent to $\Lambda=\on{Gr}(\Mult_\G)$. 
For instance, the invariance of $\Lambda$ under the (local) flow of $X^1_\sigma$ follows from the fact that $g_1\circ g_2=(g_1\circ h^{-1})\circ (h\circ g_2)$ whenever $\sz(h)=\tz(g_2)$. The vector fields satisfy bracket relations 
\[ [X_\sigma^i,\,X_\tau^j]=X_{[\sigma,\tau]}^i\ \delta_{i,j}\]
for $\sigma,\tau\in \Gamma(A)$ and $i,j=0,1,2$. If $\G$ is $\tz$-connected, then the graph is generated from $M_\Delta^{[2]}\subset \Lambda$ (the units, embedded as $m\mapsto (m,m,m)$) by the flow of these vector fields. In fact, it is already obtained using the flows of the $X_\sigma^i$'s for $i=1,2$ (or $i=0,2$, or $i=0,1$). 
For reference in Section \ref{sec:karasevweinstein}, let us note the following partial converse. 
\begin{proposition}
Let $(A,\a,[\cdot,\cdot])$ be a Lie algebroid over $M$. Suppose $i\colon M\to P$ is an embedding, with normal bundle $\nu(P,M)\cong A$, and suppose $\sigma^L,\ \sigma^R\in \mf{X}(P)$ are vector fields on $P$, mapping to $\sigma\in \Gamma(A)$ under the quotient map $TP|_M\to A$, with $\a(\sigma)\sim_i \sigma^L-\sigma^R$, and 
satisfying the bracket relations \eqref{eq:bracketrelations}. Then  a neighborhood of $M$ in $P$ inherits a structure of a local Lie groupoid integrating $A$, in such a way that $\sigma^L,\sigma^R$ are the left, right invariant vector fields. 
\end{proposition}
\begin{proof}[Sketch of proof] 
Since $\sigma^L$ maps to $\sigma$, the restrictions of the left-invariant vector fields 
to $M$ span a complement to $TM$ in $TP|_M$. In a particular, on a neighborhood of 
$M$ they determine a distribution of rank equal to that of $A$. By the bracket relations, 
this distribution is Frobenius integrable and invariant under that $\sigma^R$'s. A similar argument applies to the vector fields $-\sigma^R$. Taking $P$ smaller if necessary, we can assume that these define fibrations $\tz,\sz\colon P\to M$, with $\sigma^L,\sigma^R$ tangent to the respective fibers. 

Define vector fields $X_\sigma^i$ in $P\times P\times P$, as above, and let $M_\Delta^{[2]}\subset P\times P\times P$ be the triagonal embedding of $M$. Along the submanifold $M_\Delta^{[2]}$, hence also on some neighborhood the vector fields $X_\sigma^0,\,X_\tau^2$ span a subbundle of $TP$ of rank equal to twice the rank of $A$. The bracket relations guarantee that this distribution is integrable,
hence they define a foliation, and since the intersection of this subbundle with $TM_\Delta^{[2]}$ is trivial, we conclude that 
the flow-out of $M_\Delta$ under these vector fields defines a (germ of a) submanifold $\Lambda\subset P\times P\times P$. By construction, $\Lambda$ contains $M_\Delta^{[2]}$, and is invariant under $X_\sigma^0,\,X_\tau^2$. In fact, it is also invariant under $X_\sigma^1$, since these vector fields commute with $X_\sigma^0,X_\sigma^2$, and their restriction to $M_\Delta$  lies in subbundle spanned by $TM_\Delta^{[2]}$ together with the span of these vector fields. 

Under projection of $P\times P\times P$ to the last two factors, $X_\sigma^0,X_\sigma^2$ are related to the vector fields 
$(-\sigma^R,0)$ and $(0,\sigma^L)$, respectively. Hence, this projection restricts to a diffeomorphism from 
$\Lambda$ onto a neighborhood of the diagonal embedding of $M$ in 
$P^{(2)}=P\ _{\sz}\times_{\tz} P\subset P\times P$. It follows that $\Lambda$ is the graph of a 
multiplication map $\circ\colon P^{(2)}\to P$, defined on some neighborhood of $M$ in $P^{(2)}$. 
Letting  $\on{id}_P\colon P\to P$ be the identity relation (given by the diagonal 
in $P\times \ol{P}$, the associativity of the groupoid multiplication means that 
\[ \Lambda\circ (\Lambda\times \on{id}_P)=\Lambda\circ (\on{id}_P\times \Lambda)\]
as relations $P\times P\times P\da P$, where the circle means composition of relations. In fact, we can see that both sides 
are given by  
\[ \Lambda^{[2]}\subset P\times (P\times P\times P),\]
the submanifold generated from the diagonal $M_\Delta^{[3]}$ (consisting of elements $(m,m,m,m)$) by the action of vector fields of the form
\[ (-\sigma^R,-\sigma^R,0,0),\ (0,\sigma^L,-\sigma^R,0),\ (0,0,\sigma^L,-\sigma^R).\qedhere\]
\end{proof}

\section{Proof of Lemma \ref{lem:technical}}\label{app:tech}
The formulas from Lemma \ref{lem:technical} are proved by `straightforward' calculation. By definition, 
\[ \iota(\varrho(\zeta))\theta^L=\Ad_{g^{-1}}\pr_\g(\Ad_g\zeta).\] 
Hence, using the Maurer-Cartan equation 
$\d\theta^L=-\hh[\theta^L,\theta^L]$, 
\begin{align*}
\L(\varrho(\zeta))\theta^L&=\d \iota(\varrho(\zeta))\theta^L+\iota(\varrho(\zeta))\d\theta^L\\
&=\d\big(
\Ad_{g^{-1}}\pr_\g(\Ad_g \zeta)\big)-[\Ad_{g^{-1}}\pr_\g (\Ad_g \zeta),\theta^L]\\
&=-[\theta^L,\Ad_{g^{-1}}\pr_\g(\Ad_g \zeta)]+\Ad_{g^{-1}}\pr_\g(\Ad_g[\theta^L,\zeta])
 -[\Ad_{g^{-1}}\pr_\g (\Ad_g \zeta),\theta^L]\\
&=\Ad_{g^{-1}}\pr_\g([\Ad_g \theta^L,\Ad_g \zeta]).
\end{align*} 
Using this identity we calculate, for $\zeta_1,\zeta_2\in\dd$, 
\begin{align*}
\iota([\varrho(\zeta_1),\varrho(\zeta_2)])\theta^L&=
\L_{\varrho(\zeta_1)}\iota_{\varrho(\zeta_2)}\theta^L-\iota_{\varrho(\zeta_2)}\L_{\varrho(\zeta_1)}\theta^L\\
&=\iota_{\varrho(\zeta_1)}\L_{\varrho(\zeta_2)}\theta^L
-\iota_{\varrho(\zeta_2)}\L_{\varrho(\zeta_1)}\theta^L
-\iota_{\varrho(\zeta_1)}\iota _{\varrho(\zeta_2)}\d \theta^L
\\
&=\Ad_{g^{-1}}\pr_\g\Big([\pr_\g(\Ad_g \zeta_1),\,\Ad_g\zeta_2])
-[\pr_\g(\Ad_g \zeta_2),\,\Ad_g\zeta_1]\Big)\\
&\ \ \ -[\pr_\g(\Ad_g\zeta_1),\ \pr_\g(\Ad_\g\zeta_2)]\Big)\\
&=\Ad_{g^{-1}}\pr_\g[\Ad_g\zeta_1,\Ad_g\zeta_2])\\
&= \iota(\varrho([\zeta_1,\zeta_2])\theta^L,
\end{align*} 
proving that $\varrho$ is a Lie algebra morphism. The 1-form part of $\Cour{\mathsf{e}(\zeta_1),\mathsf{e}(\zeta_2)}$
is $\L_{\varrho(\zeta_1)}\l\theta^L,\zeta_2\r-\iota_{\varrho(\zeta_2)} \d \l\theta^L,\zeta_1\r$. We have
\begin{align*}
\L_{\varrho(\zeta_1)}\l\theta^L,\zeta_2\r&=
\l \pr_\g([\Ad_g \theta^L,\Ad_g \zeta_1]),\,\Ad_g\zeta_2\r\\
&=
\l [\Ad_g \theta^L,\Ad_g \zeta_1],\,\pr_\h(\Ad_g\zeta_2)\r\\
&=\l \Ad_g\theta^L,\,[\Ad_g \zeta_1,\,\pr_\h(\Ad_g\zeta_2)]\r,
\end{align*}
while 
\begin{align*}
-\iota_{\varrho(\zeta_2)}\d\l\theta^L,\zeta_1\r&=\l [\iota_{\varrho(\zeta_2)}\theta^L,\theta^L]\,\zeta_1\r\\
&= \l [\pr_\g(\Ad_g \zeta_2),\Ad_g \theta^L],\Ad_g\zeta_1\r\\
&=\l \Ad_g\theta^L,\,[\Ad_g \zeta_1,\,\pr_\g(\Ad_g\zeta_2)]\r.
\end{align*}
We conclude 
\[ \L_{\varrho(\zeta_1)}\l\theta^L,\zeta_2\r-\iota_{\varrho(\zeta_2)} \d \l\theta^L,\zeta_1\r=
\l \Ad_g\theta^L,\,[\Ad_g \zeta_1,\,\Ad_g\zeta_2)]\r=\l\theta^L,\,[\zeta_1,\zeta_2]\r,
\]
completing the proof that $\mathsf{e}\colon \dd\to \Gamma(\T G)$ preserves brackets. The fact that it also preserves metrics is left as an exercise.
\end{appendix}

\bibliographystyle{amsplain}

\begin{thebibliography}{10}

\bibitem{al:po}
A.~Alekseev, \emph{On {P}oisson actions of compact {L}ie groups on symplectic
  manifolds}, J.~Differential Geom. \textbf{45} (1997), no.~2, 241--256.

\bibitem{al:gw}
A.~Alekseev and E.~Meinrenken, \emph{Ginzburg-{W}einstein via
  {G}elfand-{Z}eitlin}, J. Differential Geom. \textbf{76} (2007), no.~1, 1--34.

\bibitem{al:lin}
\bysame, \emph{{Linearization of Poisson Lie group structures}}, Journal of
  Symplectic Geometry \textbf{14} (2016), 227--267.

\bibitem{be:tr}
V.~G. Belavin, A.~A.~and~Drinfel{\cprime}d, \emph{Triangle equations and simple
  {L}ie algebras}, Mathematical physics reviews, Vol.~4, Harwood Academic
  Publ., Chur, 1984, Translated from the Russian, pp.~93--165.

\bibitem{ber:rem}
F.~A. Berezin, \emph{Several remarks on the associative hull of a {L}ie
  algebra}, Funkcional. Anal. i Prilo\v zen \textbf{1} (1967), no.~2, 1--14.

\bibitem{bro:sym}
D.~Broka and P.~Xu, \emph{{Symplectic realizations of holomorphic Poisson
  manifolds}}, Preprint, 2015, arXiv:1512.08847.

\bibitem{bu:ga}
H.~Bursztyn, \emph{On gauge transformations of {P}oisson structures}, Quantum
  field theory and noncommutative geometry, Lecture Notes in Phys., vol. 662,
  Springer, Berlin, 2005, pp.~89--112.

\bibitem{bur:lin}
H.~Bursztyn, A.~Cabrera, and C.~Ortiz, \emph{Linear and multiplicative
  2-forms}, Lett. Math. Phys. \textbf{90} (2009), no.~1-3, 59--83.

\bibitem{bur:spl}
H.~Bursztyn, H.~Lima, and E.~Meinrenken, \emph{{Splitting theorems for Poisson
  and related structures}}, J. Reine Angew.~Math.~ (to appear),
  arXiv:1605.05386.

\bibitem{bur:gauge}
H.~Bursztyn and O.~Radko, \emph{Gauge equivalence of dirac structures and
  symplectic groupoids}, Ann.~Inst.~Fourier (Grenoble) \textbf{53} (2003),
  309--337.

\bibitem{cab:con}
A.~Cabrera, I.~Marcut, and A.~Salazar, \emph{{A construction of local Lie
  groupoids using Lie algebroid sprays}}, Preprint, 2017.

\bibitem{cat:poi}
A.~Cattaneo and G.~Felder, \emph{Poisson sigma models and symplectic
  groupoids}, Quantization of singular symplectic quotients, Progr. Math., vol.
  198, Birkh\"auser, Basel, 2001, pp.~61--93.

\bibitem{cat:coi}
A.~S. Cattaneo and M.~Zambon, \emph{Coisotropic embeddings in {P}oisson
  manifolds}, Trans. Amer. Math. Soc. \textbf{361} (2009), no.~7, 3721--3746.

\bibitem{cos:gro}
A.~Coste, P.~Dazord, and A.~Weinstein, \emph{Groupo\"\i des symplectiques},
  Publications du {D}\'epartement de {M}ath\'ematiques. {N}ouvelle {S}\'erie.
  {A}, {V}ol.\ 2, Publ. D\'ep. Math. Nouvelle S\'er. A, vol.~87, Univ.
  Claude-Bernard, Lyon, 1987, pp.~i--ii, 1--62.

\bibitem{cou:di}
T.~Courant, \emph{Dirac manifolds}, Trans.~Amer.~Math.~Soc. \textbf{319}
  (1990), no.~2, 631--661.

\bibitem{couwein:beyond}
T.~Courant and A.~Weinstein, \emph{Beyond {P}oisson structures}, Action
  hamiltoniennes de groupes.~Troisi\`eme th\'eor\`eme de {L}ie (Lyon, 1986),
  Travaux en Cours, vol.~27, Hermann, Paris, 1988, pp.~39--49.

\bibitem{cra:intlie}
M.~Crainic and R.~Fernandes, \emph{Integrability of {L}ie brackets}, Ann. of
  Math. (2) \textbf{157} (2003), no.~2, 575--620.

\bibitem{cra:intpoi}
\bysame, \emph{Integrability of {P}oisson brackets}, J. Differential Geom.
  \textbf{66} (2004), no.~1, 71--137.

\bibitem{cra:exi}
M.~Crainic and I.~Marcut, \emph{On the existence of symplectic realizations},
  J. Symplectic Geom. \textbf{9} (2011), no.~4, 435--444.

\bibitem{ca:ge}
A.~Cannas da~Silva and A.~Weinstein, \emph{Geometric models for noncommutative
  algebras}, American Mathematical Society, Providence, RI, 1999.

\bibitem{dor:dir}
I.~Ya. Dorfman, \emph{Dirac structures and integrability of nonlinear evolution
  equations}, Nonlinear Science - theory and applications, Wiley, Chichester,
  1993.

\bibitem{dr:ha}
V.~G. Drinfel{\cprime}d, \emph{Hamiltonian structures on {L}ie groups, {L}ie
  bialgebras and the geometric meaning of classical {Y}ang-{B}axter equations},
  Dokl.~Akad.~Nauk SSSR \textbf{268} (1983), no.~2, 285--287.

\bibitem{dri:poi}
V.~G. Drinfel{\cprime}d, \emph{On {P}oisson homogeneous spaces of
  {P}oisson-{L}ie groups}, Teoret. Mat. Fiz. \textbf{95} (1993), no.~2,
  226--227.

\bibitem{duf:po}
J.-P. Dufour and N.T. Zung, \emph{Poisson structures and their normal forms},
  Progress in Mathematics, vol. 242, Birkh\"auser Verlag, Basel, 2005.

\bibitem{kos:his}
Y.~Kosmann-Schwarzbach (ed.), \emph{{Sim\'{e}on-Denis Poisson: Les
  math\'{e}matiques au service de la science}}, Les \'{E}ditions de l'\'{E}cole
  polytechnique, 2013.

\bibitem{fre:nor}
P.~Frejlich and I.~Marcut, \emph{The local normal form around {P}oisson
  transversals}, Pacific J. Math.(to appear), arXiv1306.6055.

\bibitem{fre:rea}
\bysame, \emph{{On dual pairs in Dirac geometry}}, preprint (2016),
  arXiv:1602.02700.

\bibitem{fuc:lie}
B.~Fuchssteiner, \emph{The lie algebra structure of degenerate hamiltonian and
  bi- hamiltonian systems}, Prog. Theor. Physics \textbf{68} (1982),
  1082--1104.

\bibitem{gi:lp}
V.~L. Ginzburg and A.~Weinstein, \emph{Lie-{P}oisson structure on some
  {P}oisson {L}ie groups}, J.~Amer.~Math.~Soc. \textbf{5} (1992), no.~2,
  445--453.

\bibitem{gua:ge1}
M.~Gualtieri, \emph{Generalized complex geometry}, Ann. of Math. (2)
  \textbf{174} (2011), no.~1, 75--123.

\bibitem{hu:ham}
S.~Hu, \emph{Hamiltonian symmetries and reduction in generalized geometry},
  Houston J. Math \textbf{35} (2009), no.~3, 787--811.

\bibitem{jot:dir}
M.~Jotz, \emph{Dirac {L}ie groups, {D}irac homogeneous spaces and the theorem
  of {D}rinfeld}, Indiana Univ. Math. J. \textbf{60} (2011), no.~1, 319--366.

\bibitem{kar:ana}
M.~V. Karas{\"e}v, \emph{Analogues of objects of the theory of {L}ie groups for
  nonlinear {P}oisson brackets}, Izv. Akad. Nauk SSSR Ser. Mat. \textbf{50}
  (1986), no.~3, 508--538, 638.

\bibitem{kos:cro}
J.~L. Koszul, \emph{Crochet de {S}chouten-{N}ijenhuis et cohomologie},
  Ast\'erisque (1985), no.~Numero Hors Serie, 257--271, The mathematical
  heritage of {\'E}lie Cartan (Lyon, 1984).

\bibitem{lau:po}
C.~Laurent-Gengoux, A.~Pichereau, and P.~Vanhaecke, \emph{Poisson structures},
  Grundlehren der Mathematischen Wissenschaften [Fundamental Principles of
  Mathematical Sciences], vol. 347, Springer, Heidelberg, 2013.

\bibitem{le:al}
S.~Levendorski and Y.~Soibelman, \emph{Algebras of functions on compact quantum
  groups, {S}chubert cells and quantum tori}, Comm.~Math.~Phys. \textbf{139}
  (1991), no.~1, 141--170.

\bibitem{lib:cou}
D.~Li-Bland and E.~Meinrenken, \emph{{C}ourant algebroids and {P}oisson
  geometry}, International Mathematics Research Notices \textbf{11} (2009),
  2106--2145.

\bibitem{lib:dir}
\bysame, \emph{{Dirac Lie groups}}, Asian Journal of Mathematics \textbf{18}
  (2014), no.~5, 779--816.

\bibitem{lib:pro}
P.~Libermann, \emph{Probl\`emes d'\'equivalence et g\'eom\'etrie symplectique},
  Third {S}chnepfenried geometry conference, {V}ol. 1 ({S}chnepfenried, 1982),
  Ast\'erisque, vol. 107, Soc. Math. France, Paris, 1983, pp.~43--68.

\bibitem{lic:var}
A.~Lichnerowicz, \emph{Les vari\'et\'es de {P}oisson et leurs alg\`ebres de
  {L}ie associ\'ees}, J. Differential Geometry \textbf{12} (1977), no.~2,
  253--300.

\bibitem{liu:ma}
Z.-J. Liu, A.~Weinstein, and P.~Xu, \emph{Manin triples for {L}ie
  bialgebroids}, J.~Differential Geom. \textbf{45} (1997), no.~3, 547--574.

\bibitem{liu:dist}
\bysame, \emph{Dirac structures and {P}oisson homogeneous spaces},
  Comm.~Math.~Phys. \textbf{192} (1998), no.~1, 121--144.

\bibitem{lu:note}
J.-H. Lu, \emph{A note on {P}oisson homogeneous spaces}, Poisson geometry in
  mathematics and physics, Contemp. Math., vol. 450, Amer. Math. Soc.,
  Providence, RI, 2008, pp.~173--198.

\bibitem{lu:poi}
Jiang-Hua Lu, \emph{Poisson homogeneous spaces and {L}ie algebroids associated
  to {P}oisson actions}, Duke Math. J. \textbf{86} (1997), no.~2, 261--304.

\bibitem{mac:gen}
K.~Mackenzie, \emph{General theory of {L}ie groupoids and {L}ie algebroids},
  London Mathematical Society Lecture Note Series, vol. 213, Cambridge
  University Press, Cambridge, 2005.

\bibitem{mac:int}
K.~Mackenzie and P.~Xu, \emph{Integration of {L}ie bialgebroids}, Topology
  \textbf{39} (2000), no.~3, 445--467.

\bibitem{mar:red}
J.~Marrero, E.~Padron, and M.~Rodriguez-Olmos, \emph{{Reduction of a
  symplectic-like {L}ie algebroid with momentum map and its application to
  fiberwise linear Poisson structures}}, Journal of Physics A: Math. Theor.
  \textbf{45} (2012), 165--201.

\bibitem{me:dir}
E.~Meinrenken, \emph{{Dirac actions and Lu's Lie algebroid}}, Transformation
  Groups, arXiv:1412.3154.

\bibitem{mel:ati}
R.~B. Melrose, \emph{{The Atiyah-Patodi-Singer index theorem}}, Research Notes
  in Mathematics, vol.~4, A K Peters, Ltd., wellesley, 1993.

\bibitem{mir:no}
E.~Miranda and N.~T. Zung, \emph{A note on equivariant normal forms of
  {P}oisson structures}, Math. Res. Lett. \textbf{13} (2006), no.~5-6,
  1001--1012.

\bibitem{moe:fol}
I.~Moerdijk and J.~Mr{\v{c}}un, \emph{Introduction to foliations and {L}ie
  groupoids}, Cambridge Studies in Advanced Mathematics, vol.~91, Cambridge
  University Press, Cambridge, 2003.

\bibitem{ort:mu}
C.~Ortiz, \emph{Multiplicative {D}irac structures on {L}ie groups}, C. R. Math.
  Acad. Sci. Paris \textbf{346} (2008), no.~23-24, 1279--1282.

\bibitem{pra:th}
J.~Pradines, \emph{Th\'eorie de {L}ie pour les groupo\"\i des
  diff\'erentiables. {C}alcul diff\'erenetiel dans la cat\'egorie des
  groupo\"\i des infinit\'esimaux}, C. R. Acad. Sci. Paris S\'er. A-B
  \textbf{264} (1967), A245--A248.

\bibitem{rob:cla}
P.~Robinson, \emph{{The classification of {D}irac homogeneous spaces}}, Thesis,
  University of Toronto, 2014. arXiv:1411.2958.

\bibitem{se:wh}
M.~A. Semenov-Tian-Shansky, \emph{What is a classical $r$-matrix?}, Funct.
  Anal. Appl. \textbf{17} (1983), 259--272.

\bibitem{se:dr}
\bysame, \emph{Dressing transformations and {P}oisson group actions}, Publ.
  Res. Inst. Math. Sci. \textbf{21} (1985), no.~6, 1237--1260.

\bibitem{sev:let}
P.~{\v{S}}evera, \emph{{Letters to Alan Weinstein, 1998-2000}}, available at
  author's website.

\bibitem{sev:poi1}
P.~{\v{S}}evera and A.~Weinstein, \emph{Poisson geometry with a 3-form
  background}, Progr. Theoret. Phys. Suppl. (2001), no.~144, 145--154,
  Noncommutative geometry and string theory (Yokohama, 2001).

\bibitem{ste:loc}
S.~Sternberg, \emph{Local contractions and a theorem of {P}oincar\'e}, Amer. J.
  Math. \textbf{79} (1957), 809--824.

\bibitem{ste:min}
\bysame, \emph{On minimal coupling and the symplectic mechanics of a classical
  particle in the presence of a {Y}ang-{M}ills field}, Proc.~ Nat.~ Acad.~
  Sci.~ USA \textbf{74} (1977), 5253--5254.

\bibitem{va:po}
I.~Vaisman, \emph{Lectures on the geometry of {P}oisson manifolds},
  Birkh\"auser, 1994.

\bibitem{wei:uni}
A.~Weinstein, \emph{A universal phase space for particles in {Y}ang-{M}ills
  fields}, Lett. Math. Phys. \textbf{2} (1978), 417?--420.

\bibitem{wei:loc}
\bysame, \emph{The local structure of {P}oisson manifolds}, J. Differential
  Geom. \textbf{18} (1983), no.~3, 523--557.

\bibitem{wei:sym}
\bysame, \emph{Symplectic groupoids and {P}oisson manifolds}, Bull. Amer. Math.
  Soc. (N.S.) \textbf{16} (1987), no.~1, 101--104.

\bibitem{wei:rem}
\bysame, \emph{Some remarks on dressing transformations}, J. Fac. Sci. Univ.
  Tokyo Sect. IA Math. \textbf{35} (1988), no.~1, 163--167.

\bibitem{xu:mor}
P.~Xu, \emph{Morita equivalence of {P}oisson manifolds}, Comm. Math. Phys.
  \textbf{142} (1991), no.~3, 493--509.

\bibitem{xu:dir}
P.~Xu, \emph{Dirac submanifolds and {P}oisson involutions}, Ann. Sci. \'Ecole
  Norm. Sup. (4) \textbf{36} (2003), no.~3, 403--430.

\end{thebibliography}

\def\cprime{$'$} \def\polhk#1{\setbox0=\hbox{#1}{\ooalign{\hidewidth
  \lower1.5ex\hbox{`}\hidewidth\crcr\unhbox0}}} \def\cprime{$'$}
  \def\cprime{$'$} \def\cprime{$'$} \def\cprime{$'$} \def\cprime{$'$}
  \def\polhk#1{\setbox0=\hbox{#1}{\ooalign{\hidewidth
  \lower1.5ex\hbox{`}\hidewidth\crcr\unhbox0}}} \def\cprime{$'$}
  \def\cprime{$'$} \def\cprime{$'$} \def\cprime{$'$} \def\cprime{$'$}
\providecommand{\bysame}{\leavevmode\hbox to3em{\hrulefill}\thinspace}
\providecommand{\MR}{\relax\ifhmode\unskip\space\fi MR }
\providecommand{\MRhref}[2]{%
  \href{http://www.ams.org/mathscinet-getitem?mr=#1}{#2}
}
\providecommand{\href}[2]{#2}

\end{document}